\documentclass[11pt,review]{elsarticle}
\usepackage[utf8]{inputenc}
\usepackage{graphicx}
\graphicspath{ {./images/} }
\usepackage{mathtools,esint}
\allowdisplaybreaks

\usepackage{srcltx}
\usepackage{eurosym}
\usepackage{amssymb}
\usepackage[all,arc]{xy}
\usepackage{enumerate}
\usepackage{mathrsfs}
\usepackage{pst-node}
\usepackage{amsmath,amsthm}
\usepackage{mathtools}
\usepackage[T1]{fontenc}
\usepackage[doublespacing]{setspace}
\usepackage[titletoc,toc,page]{appendix}

\usepackage[colorlinks,plainpages=true,pdfpagelabels,hypertexnames=true,colorlinks=true,pdfstartview=FitV,linkcolor=blue,citecolor=red,urlcolor=black,pagebackref]{hyperref}
\PassOptionsToPackage{unicode}{hyperref}
\PassOptionsToPackage{naturalnames}{hyperref}

\newcommand*\bb{\mathbb}

\newcommand *\w{^\wedge}

\newcommand*\de{\partial}

\makeatletter
\newcommand{\vo}{\vec{o}\@ifnextchar{^}{\,}{}}
\makeatother

\def\YYint#1#2#3{{\setbox0=\hbox{$#1{#2#3}{\iint}$}
    \vcenter{\hbox{$#2#3$}}\kern-.50\wd0}}

\def\XXint#1#2#3{{\setbox0=\hbox{$#1{#2#3}{\int}$}
    \vcenter{\hbox{$#2#3$}}\kern-.50\wd0}}

\makeatletter
\def\namedlabel#1#2{\begingroup
   \def\@currentlabel{#2}%
   \label{#1}\endgroup
}
\makeatother
\makeatletter
\newcommand{\rmh}[1]{\mathpalette{\raisem@th{#1}}}
\newcommand{\raisem@th}[3]{\hspace*{-1pt}\raisebox{#1}{$#2#3$}}
\makeatother

\newcommand{\descitem}[2]{\item[{(#1):}]\label{#2}}
\newcommand{\descref}[2]{\hyperref[#1]{\textnormal{\textcolor{black}{(}\textcolor{blue}{\bf #2}\textcolor{black}{)}}}}

\newcommand{\dref}[2]{\hyperref[#1]{\textcolor{black}{(}\textcolor{blue}{\bf #2}\textcolor{black}{)}}}

\newcommand{\tk}{\tilde{k}}

\newcommand{\tw}{\tilde{w}}

\newcommand\RR{\mathbb{R}}

\newcommand\NN{\mathbb{N}}

\newcommand{\ve}{\varepsilon}

\newcommand{\tht}{\theta}

\newcommand{\Om}{\Omega}

\usepackage[ddmmyyyy]{datetime}
\usepackage[margin=2cm]{geometry}
\makeatletter
\g@addto@macro\normalsize{%
  \setlength\abovedisplayskip{2pt}
  \setlength\belowdisplayskip{2pt}
  \setlength\abovedisplayshortskip{4pt}
  \setlength\belowdisplayshortskip{4pt}
}
\numberwithin{equation}{section}
\everymath{\displaystyle}
\usepackage[capitalize,nameinlink]{cleveref}
\crefname{section}{Section}{Sections}
\crefname{subsection}{Subsection}{Subsections}
\crefname{condition}{Condition}{Conditions}
\crefname{hypothesis}{Hypothesis}{Hypothesis}
\crefname{assumption}{Assumption}{Assumptions}
\crefname{lemma}{Lemma}{Lemmas}
\crefname{claim}{Claim}{Claims}
\crefname{remark}{Remark}{Remarks}

\crefformat{equation}{\textup{#2(#1)#3}}
\crefrangeformat{equation}{\textup{#3(#1)#4--#5(#2)#6}}
\crefmultiformat{equation}{\textup{#2(#1)#3}}{ and \textup{#2(#1)#3}}
{, \textup{#2(#1)#3}}{, and \textup{#2(#1)#3}}
\crefrangemultiformat{equation}{\textup{#3(#1)#4--#5(#2)#6}}%
{ and \textup{#3(#1)#4--#5(#2)#6}}{, \textup{#3(#1)#4--#5(#2)#6}}%
{, and \textup{#3(#1)#4--#5(#2)#6}}

\Crefformat{equation}{#2Equation~\textup{(#1)}#3}
\Crefrangeformat{equation}{Equations~\textup{#3(#1)#4--#5(#2)#6}}
\Crefmultiformat{equation}{Equations~\textup{#2(#1)#3}}{ and \textup{#2(#1)#3}}
{, \textup{#2(#1)#3}}{, and \textup{#2(#1)#3}}
\Crefrangemultiformat{equation}{Equations~\textup{#3(#1)#4--#5(#2)#6}}%
{ and \textup{#3(#1)#4--#5(#2)#6}}{, \textup{#3(#1)#4--#5(#2)#6}}%
{, and \textup{#3(#1)#4--#5(#2)#6}}

\crefdefaultlabelformat{#2\textup{#1}#3}

\newtheorem{theorem}{Theorem}[section]
\newtheorem{lemma}[theorem]{Lemma}

\newtheorem{definition}[theorem]{Definition}
\newtheorem{remark}[theorem]{Remark}        

\numberwithin{equation}{section}

\usepackage{enumitem}
\newlist{steps}{enumerate}{1}
\setlist[steps, 1]{label = \textcolor{Cerulean}{Step \arabic*:}}
\makeatletter
\def\ps@pprintTitle{%
	\let\@oddhead\@empty
	\let\@evenhead\@empty
	\def\@oddfoot{}%
	\let\@evenfoot\@oddfoot}
\makeatother
\DeclarePairedDelimiterX{\inp}[2]{\langle}{\rangle}{#1, #2}
\newcommand{\norm}[1]{\left\lVert#1\right\rVert}

\usepackage{doi}
\begin{document}

\begin{frontmatter}
\title{Local boundedness of variational solutions to nonlocal double phase parabolic equations}
\author[myaddress]{Harsh Prasad\corref{mycorrespondingauthor}}
\ead{harsh@tifrbng.res.in}

\author[myaddress]{Vivek Tewary}
\ead{vivektewary@gmail.com and vivek2020@tifrbng.res.in}


\address[myaddress]{Tata Institute of Fundamental Research, Centre for Applicable Mathematics, Bangalore, Karnataka, 560065, India}
\cortext[mycorrespondingauthor]{Corresponding author}
\begin{abstract}
    We prove local boundedness of variational solutions to the double phase equation 
    		    \begin{align*}
        \de_t u + P.V.\int_{\mathbb{R}^N}& \frac{|u(x,t)-u(y,t)|^{p-2}(u(x,t)-u(y,t))}{|x-y|^{N+ps}}\\
        &+a(x,y)\frac{|u(x,t)-u(y,t)|^{q-2}(u(x,t)-u(y,t))}{|x-y|^{N+qs'}} \,dy = 0,
   		    \end{align*}	 
   	 under the restrictions $s,s'\in (0,1),\, 1 < p \leq q \leq p\,\frac{2s+N}{N}$ and the non-negative function $(x,y)\mapsto a(x,y)$ is assumed to be measurable and bounded. 
\end{abstract}

\end{frontmatter}
\begin{singlespace}
\tableofcontents
\end{singlespace}

\section{Introduction}

\subsection{The problem}
In this article, we are interested to study regularity theory for a double phase nonlocal parabolic equations where the nonlocal operator is modelled on the fractional $p$-Laplacian. In particular, we prove local boundedness for the following double phase equation.
\begin{align}\label{maineq}
\de_t u + P.V.\int_{\mathbb{R}^N} & \frac{|u(x,t)-u(y,t)|^{p-2}(u(x,t)-u(y,t))}{|x-y|^{N+ps}}\nonumber\\
&+a(x,y)\frac{|u(x,t)-u(y,t)|^{q-2}(u(x,t)-u(y,t))}{|x-y|^{N+qs'}} \,dy = 0.
\end{align}

However, in order to make sense of a weak solution for the above equation, we require the membership $u\in L^q(0,T;W^{s',q}(\RR^N))$. This is not guaranteed by existence theory. A suitable notion to study such evolution problems is that of variational solutions which was introduced by Lichnewsky and Temam in~\cite{lichnewskyPseudosolutionsTimedependentMinimal1978} for evolutionary minimal surface equations. This notion was extended to quasilinear parabolic equations with non-standard growth in~\cite{bogeleinExistenceEvolutionaryVariational2014}. In the companion paper~\cite{prasadExistenceVariationalSolutions2021a}, the present authors have extended the notion to nonlocal parabolic equations satisfying only a growth condition from below through convex minimization of approximating functionals. In particular, it covers the double phase equation that we study in this paper. 

In this paper, we prove local boundedness of variational solutions under the following restrictions on the gap:
\begin{align}
    \label{restriction}
    p \leq q \leq p\,\frac{2s+N}{N}:=p_*.
\end{align}
The restriction above is due to the fractional Sobolev embedding. The proof is divided into the super-critical case, viz.,
\begin{align}
    \label{restriction1}
    \frac{2N}{2s+N} < p < \infty,
\end{align} and the sub-critical case
\begin{align}
    \label{restriction2}
    1 < p \leq \frac{2N}{2s+N},
\end{align} where the latter range requires an additional integrability assumption on the solution.

Since the variational solution $u$ may not belong to $L^q(0,T;W^{s',q}(\RR^N))$, it is not clear whether Steklov averages can be used to obtain the Caccioppoli inequality. For this reason, we assume that $\de_t u\in L^2(\Om_T)$, which is guaranteed for time-independent initial-boundary data, as proved in~\cite{prasadExistenceVariationalSolutions2021a}. 

There are two main novelties to this paper. Firstly, our approach is purely variational. Local boundedness and H\"older regularity for minimizers of functionals satisfying certain growth conditions was proved by Giaquinta and Giusti in \cite{giaquintaRegularityMinimaVariational1982}. Their method was extended to the nonlocal case by Cozzi in \cite{cozziRegularityResultsHarnack2017}. We prove a Caccioppoli inequality for parabolic minimizers of a fractional evolution equation. Secondly, the variational approach is a natural framework to consider problems with non-standard growth.  

\subsection{Background}

The regularity theory of fractional $p$-Laplace equations and their parabolic counterparts has seen a great surge ever since the De Giorgi-Nash-Moser theory was extended to the nonlocal case in~\cite{dicastroLocalBehaviorFractional2016} and \cite{cozziRegularityResultsHarnack2017}. Other relevant papers include~\cite{brascoHigherSobolevRegularity2017,brascoHigherHolderRegularity2018,brascoContinuitySolutionsNonlinear2021,kuusiFractionalGehringLemma2014}

The regularity theory of $p,q$ growth problems was started by Marcellini in a series of novel papers~\cite{marcelliniRegularityMinimizersIntegrals1989,marcelliniRegularityExistenceSolutions1991,marcelliniRegularityEllipticEquations1993,marcelliniRegularityScalarVariational1996}. There is a large body of work dealing with problems of $(p,q)$-growth as well as other nonstandard growth problems, for which we point to the surveys in \cite{marcelliniRegularityGeneralGrowth2020,mingioneRecentDevelopmentsProblems2021}. Local boundedness for minimizers of functionals satisfying a non-standard growth condition was studied in~\cite{moscarielloHolderContinuityMinimizers1991,fuscoLocalBoundednessMinimizers1990,marcelliniRegularityMinimizersIntegrals1989}. These results were made sharp in~\cite{hirschGrowthConditionsRegularity2020}. For parabolic equations, comparable results may be found in~\cite{yuBoundednessSolutionsParabolic1997,singerLocalBoundednessVariational2016}. For nonlocal equations of nonstandard growth, we refer to \cite{scottSelfimprovingInequalitiesBounded2020,byunOlderRegularityWeak2021,chakerRegularitySolutionsAnisotropic2020,chakerRegularityEstimatesFractional2021,chakerNonlocalOperatorsSingular2020,chakerLocalRegularityNonlocal2021,chakerRegularityNonlocalProblems2021}. 

\subsection{Definition}

\begin{definition}
    Let $\Om$ be an open bounded subset of $\RR^N$. Let the function $H$ denote 
    \begin{align}
        H(x,y,u(x,t)-u(y,t))=\frac{|u(x,t)-u(y,t)|^p}{|x-y|^{ps}}+a(x,y)\frac{|u(x,t)-u(y,t)|^q}{|x-y|^{qs'}},
    \end{align} where $s,s'\in (0,1), p,q\in [1,\infty)$, and the measurable bounded function $(x,y)\mapsto a(x,y)$ satisfies $0\leq a\leq M<\infty$.
    Let the time-independent Cauchy-Dirichlet data $u_0$ satisfy
    \begin{align}\label{datahypo}
        u_0\in W^{s,p}(\mathbb{R}^N), u_{|\Omega}\in L^2(\Omega)\mbox{ and }\iint\limits_{\mathbb{R}^N\times\mathbb{R}^N}\frac{H(x,y,u_0(x)-u_0(y))}{|x-y|^N}\,dx\,dy<\infty.
    \end{align}

     By a variational solution to~\cref{maineq} we mean a function $u$ such that \[u\in L^p(0,T;W^{s,p}(\mathbb{R}^N))\cap C^0(0,T;L^2(\Om))\mbox{ and }u-u_0\in L^p(0,T;W^{s,p}_0(\Om))\mbox { and }\]
     \begin{align}
         \label{defvar}
         \int_{0}^{T} \int_{\Om}{\de_t v}\,{(v-u)}\,dx\,dt & + \int_{0}^{T}\iint\limits_{\RR^N\times\RR^N} \frac{H(x,y,v(x,t)-v(y,t))-H(x,y,u(x,t)-u(y,t))}{|x-y|^N}\,dx\,dy\,dt\nonumber\\
         &\geq \frac{1}{2}\norm{(v-u)(\cdot,\tau)}^2_{L^2(\Om)}-\frac{1}{2}\norm{(v-u)(\cdot,0)}^2_{L^2(\Om)},
     \end{align} for all $v\in L^p(0,T;W^{s,p}(\mathbb{R}^N))$ and $\de_t v\in L^2(0,T;\Om)$ such that $v-u_0\in L^p(0,T;W_0^{s,p}(\Om))$.
\end{definition}

Moreover, in~\cite{prasadExistenceVariationalSolutions2021a}, we have shown that variational solutions are parabolic minimizers, which we define below. We shall use the more convenient definition of parabolic minimizers in order to prove local boundedness.

\begin{definition}\label{defparmin}
    A measurable function $u:\RR^N\times(0,\infty)\to\RR$ is called a parabolic minimizer for the equation \cref{maineq} if $u\in L^p(0,T;W^{s,p}(\RR^N))$ and $u-u_0\in L^p(0,T;W^{s,p}_0(\Omega))$, and the following minimality condition holds:
    \begin{align*}
        \int_0^T \int_{\Om} u \de_t\phi\,dx\,dt &+ \int_0^T\iint\limits_{\mathbb{R}^N\times\mathbb{R}^N} \frac{H(x,y,u(x,t)-u(y,t))}{|x-y|^N}\,dx\,dy\,dt\\
        &\qquad\leq \int_0^T\iint\limits_{\mathbb{R}^N\times\mathbb{R}^N} \frac{H(x,y,(u+\phi)(x,t)-(u+\phi)(y,t))}{|x-y|^N}\,dx\,dy\,dt,
    \end{align*} whenever $\phi\in C_0^\infty(\Omega\times(0,T))$
    \end{definition}

\begin{remark}
In our definition of variational solution, both the solution and the comparison map have to match on $\Omega^c\times(0,T)$ and since we have assumed that the data on $\Omega^c$ is in $W^{s,p}(\bb{R}^N)$ we may cancel integrals of $H$ over $\Omega^c\times\Omega^c$ on both sides to obtain the following equivalent form of the variational inequality: 
\[
\begin{split}
    \int_0^T \iint\limits_{C_{\Omega}}\frac{H(x,y,u(x,t)-u(y,t))}{|x-y|^N} \,dx \,dy \,dt \leq \int_0^T \iint\limits_{C_{\Omega}}\frac{H(x,y,v(x,t)-v(y,t))}{{|x-y|^N}} \,dx \,dy \,dt \\ + \int_0^T \int_{\Omega}\de_tv\cdot(v-u) + \frac{1}{2}\norm{v(\cdot,0) - u_0(\cdot)}_{L^2(\Omega)}^2 - \frac{1}{2}\norm{v(\cdot,T) - u(\cdot,T)}_{L^2(\Omega)}^2
\end{split}
\]
where $C_{\Omega} = (\Omega^c\times\Omega^c)^c$. The same goes for parabolic minimizers.
\end{remark}

\subsection{Main Results}

We prove the following local boundedness result in the super-critical case \cref{restriction1}. 

\begin{theorem}{\textbf{(Local boundedness in the supercritical case)}}
    \label{mainthm1}

    Let $u\in L^p(0,T;W^{s,p}(\RR^N))\cap C^0(0,T;L^2(\Om))$
    be a variational solution to~\cref{maineq} in accordance with the definition~\cref{defvar} with the initial-boundary data $u_0$ satisfying~\cref{datahypo}. Further assume that
    \[u\in L^\infty(0,T;L^{p-1}_{sp}(\RR^N))\cap L^\infty(0,T;L^{q-1}_{qs'}(\RR^N)).\] The variational solution $u$ is locally bounded in $\Om_T$ provided the restrictions~\cref{restriction} and \cref{restriction1} holds. In the case $q<p_*$, we obtain the following estimate for any $\delta\in (q,p_*)$ and $\sigma\in [1/2,1)$,
    \begin{align}\label{boundest}
        \underset{Q_{\sigma R,\sigma R^{sp}}}{\text{ess sup}} \, u  \,\leq\, &\text{Tail}_{p,s,\infty}(u_+;x_0,R/2,t_0-R^{sp},t_0) + \text{Tail}_{q,s',\infty}(u_+;x_0,R/2,t_0-R^{sp},t_0)\nonumber\\
        & + C 2^{\frac{bN^2\kappa^2}{s\delta\tau(N\kappa+s\delta)}} \tilde{\mathcal{B}}^{\frac{N\kappa}{\tau(N\kappa+s\delta)}}\mathcal{A}^{\frac{1}{\tau}} \left(\fint\limits_{t_0-R^{sp}}^{t_0}\fint\limits_{B_R(x_0)} u_+^\delta \right)^{\frac{s\delta}{\tau(N\kappa+s\delta)}}\wedge 1 ,
    \end{align} where $\kappa=1+\dfrac{2s}{N}$, $b=\kappa\left(N+\max\{ps,qs'\}+\delta\right)$, $\tau=\min\{\delta-q,\delta-p,\delta-2\}$ and $\mathcal{A}=2+R^{sp-qs'}$, $\tilde{\mathcal{B}}= \left[\frac{1}{\sigma^{ps}(1-\sigma)^{N+ps}}+\frac{1}{\sigma^{qs'}(1-\sigma)^{N+qs'}}+\frac{1}{(1-\sigma)^{q}}\right]^{\left(1+\frac{sp}{N}\right)\left(\frac{\delta}{p\kappa}\right)}$, and $C$ is a universal constant.
\end{theorem}

We also prove the following local boundedness result in the sub-critical case \cref{restriction2} under an assumption of higher integrability on the solution. 

\begin{theorem}{\textbf{(Local boundedness in the subcritical case)}}
    \label{mainthm2}

    Let $u\in L^p(0,T;W^{s,p}(\RR^N))\cap C^0(0,T;L^2(\Om))$
    be a variational solution to~\cref{maineq} in accordance with the definition~\cref{defvar} with the initial-boundary data $u_0$ satisfying~\cref{datahypo}. Assume that one of the two following two hypotheses is true.
    \begin{description}
        \descitem{H1}{H1} $u\in L^r_{\text{loc}}(\Om_T)$ for some 
        \begin{align}\label{restiction3} 
        r>\max\left\{2,\frac{N(2-q)}{sq}+\frac{4}{q},\frac{N(2-p)}{sp}\right\}.
        \end{align}
        \descitem{H2}{H2} There is a sequence $(u_k)_{k\in\NN}$ of bounded variational solutions of \cref{maineq} such that $u_k\to u$ in $L^r_{\text{loc}}(\Om_T)$ for some $r>\max\left\{2,\frac{N(2-p)}{sp}\right\}$ and $u_k$ is uniformly bounded in $L^\infty(0,T;L^{p-1}_{sp}(\RR^N))\cap L^\infty(0,T;L^{q-1}_{qs'}(\RR^N))$. This hypothesis holds, for example, if we know {\it a priori} that $u$ is qualitatively bounded.
    \end{description}
    Further assume that
    \[u\in L^\infty(0,T;L^{p-1}_{sp}(\RR^N))\cap L^\infty(0,T;L^{q-1}_{qs'}(\RR^N)).\] The variational solution $u$ is locally bounded in $\Om_T$ in the ranges \cref{restriction} and \cref{restriction2}. In the case $q<p_*$, we obtain the following estimate
    \begin{align}\label{boundest2}
            \underset{Q_{R/2,(R/2)^{sp}}}{\text{ess sup}} \, u  \,\leq\, &\text{Tail}_{p,s,\infty}(u_+;x_0,R/2,t_0-R^{sp},t_0) + \text{Tail}_{q,s',\infty}(u_+;x_0,R/2,t_0-R^{sp},t_0)\nonumber\\
            &\qquad\qquad+C\left(\fint\fint\limits_{Q_{R,R^{sp}}}u^r\,dx\,dt\right)^{\frac{sp}{(r-2)(N+sp)-N(r-p_*)}},
    \end{align} where $\kappa=1+\dfrac{2s}{N}$, $b=\kappa\left(N+\max\{ps,qs'\}+\delta\right)$, $\tau=\min\{\delta-q,\delta-p,\delta-2\}$ and $\mathcal{A}=2+R^{sp-qs'}$ and $C$ is a constant depending on the data.
\end{theorem}

\subsection{Plan of the paper} The paper is organized as follows. In \cref{sec2}, we state the preliminaries required for later sections, including the notation and embeddings of the functional spaces as well as iteration lemmas. In \cref{sec3}, we prove a Caccioppoli inequality for parabolic minimizers of nonlocal parabolic equations. In \cref{sec4}, we prove certain recursive inequalities and the boundedness estimate in the ranges $q<p_*$ and \cref{restriction1}. \cref{recurselemma} does not require the restriction $p>\frac{2N}{2s+N}$ and therefore, it will also be used in the subcritical case. \cref{sec5} is devoted to the qualitative boundedness in the limit case $q=p_*$ for the supercritical case. We finally prove the boundedness estimate for $q<p_*$ as well as the qualitative boundedness in the limit case $q=p_*$ in \cref{sec6} for the range \cref{restriction2}.

\section{Preliminaries}\label{sec2}
\subsection{Notations}
We begin by collecting the standard notation that will be used throughout the paper:
\begin{itemize}
\item We shall denote $N$ to be the space dimension. We shall denote by $z=(x,t)$ a point in $ \RR^N\times (0,T)$.  
\item We shall alternately use $\dfrac{\partial f}{\partial t},\partial_t f,f'$ to denote the time derivative of f.
 \item Let $\Omega$ be an open bounded domain in $\mathbb{R}^N$ with boundary $\partial \Omega$ and for $0<T\leq\infty$,  let $\Omega_T\coloneqq \Omega\times (0,T)$. 
 \item We shall use the notation
\begin{align*}
	&B_{\rho}(x_0)=\{x\in\RR^N:|x-x_0|<\rho\},\\
	&\overline{B}_{\rho}(x_0)=\{x\in\RR^N:|x-x_0|\leq\rho\},\\
    &I_{\tht}(t_0)=\{t\in\RR:t_0-\tht<t<t_0\},\\
    &Q_{\rho,\tht}(z_0)=B_{\rho}(x_0)\times I_\tht(t_0)
\end{align*} 
\item The maximum of two numbers $a$ and $b$ will be denoted by $a\wedge b\coloneqq \max(a,b)$. 
\item Integration with respect to either space or time only will be denoted by a single integral $\int$ whereas integration on $\Om\times\Om$ or $\RR^N\times\RR^N$ will be denoted by a double integral $\iint$.
\item The notation $a \lesssim b$ is shorthand for $a\leq C b$ where $C$ is a universal constant which only depends on the dimension $N$, exponents $p$, $q$, and the numbers $M$, $s$ and $s'$. 
\item For a function $u$ defined on the cylinder $Q_{\rho,\theta}(z_0)$ and any level $k \in \bb{R}$ we write $w_{\pm} = (u-k)_{\pm}$ 
\item We denote $A_{\pm}(k) = \{w_{\pm} > 0\}$; for any ball $B_r$ and time interval $I$ we write $A_{\pm}(k) \cap (B_{r}\times I) = A_{\pm}(k,r,I)$. 
\end{itemize}

\subsection{Function spaces}
Let $1<p<\infty$, we denote by $p'=p/(p-1)$ the conjugate exponent of $p$. Let $\Om$ be an open subset of $\RR^N$. We define the {\it Sobolev-Slobodeki\u i} space, which is the fractional analogue of Sobolev spaces.
\begin{align*}
    W^{s,p}(\Om)=\left\{ \psi\in L^p(\Omega): [\psi]_{W^{s,p}(\Om)}<\infty \right\}, s\in (0,1),
\end{align*} where the seminorm $[\cdot]_{W^{s,p}(\Om)}$ is defined by 
\begin{align*}
    [\psi]_{W^{s,p}(\Om)}=\left(\iint\limits_{\Om\times\Om} \frac{|\psi(x)-\psi(y)|^p}{|x-y|^{N+ps}}\,dx\,dy\right)^{\frac 1p}.
\end{align*}
The space when endowed with the norm $\norm{\psi}_{W^{s,p}(\Om)}=\norm{\psi}_{L^p(\Om)}+[\psi]_{W^{s,p}(\Om)}$ becomes a Banach space. The space $W^{s,p}_0(\Om)$ is the subspace of $W^{s,p}(\RR^N)$ consisting of functions that vanish outside $\Om$. We will use the notation $W^{s,p}_{(u_0)}(\Om)$ to denote the space of functions in $W^{s,p}(\RR^N)$ such that $u-u_0\in W^{s,p}_0(\Om)$.

Let $I$ be an interval and let $V$ be a separable, reflexive Banach space, endowed with a norm $\norm{\cdot}_V$. We denote by $V^*$ its topological dual space. Let $v$ be a mapping such that for a.e. $t\in I$, $v(t)\in V$. If the function $t\mapsto \norm{v(t)}_V$ is measurable on $I$, then $v$ is said to belong to the Banach space $L^p(I;V)$ if $\int_I\norm{v(t)}_V^p\,dt<\infty$. It is well known that the dual space $L^p(I;V)^*$ can be characterized as $L^{p'}(I;V^*)$.

Since the boundedness result requires some finiteness condition on the nonlocal tails, we define the tail space as below
\begin{align}\label{tailspace}
    L^m_{\alpha}(\RR^N):=\left\{ v\in L^m_{\text{loc}}(\RR^N):\int\limits_{\RR^N}\frac{|v(x)|^m}{1+|x|^{N+\alpha}}\,dx<+\infty \right\},\,m>0,\,\alpha>0.
\end{align}

Then a nonlocal tail is defined by 
\begin{align}
    \label{NonlocalTail}
    \text{Tail}_{m,s,\infty}(v;x_0,R,I)=\text{Tail}_\infty(v;x_0,R,t_0-\tht,t_0)=\sup_{t\in (t_0-\tht, t_0)}\left( R^{sm}\int\limits_{\RR^N\setminus B_R(x_0)} \frac{|v(x,t)|^{m-1}}{|x-x_0|^{N+sm}}\,dx \right)^{\frac{1}{m-1}},
\end{align} where $(x_0,t_0)\in \RR^N\times (0,T)$ and the interval $I=(t_0-\tht,t_0)\subseteq (0,T)$. From this definition, it follows that $\text{Tail}_{m,s,\infty}(v;x_0,R,I)<\infty$ for $v\in L^\infty(0,T;L^{m-1}_{sm}(\RR^N))$. 

\subsection{Auxiliary Results}
We collect the following standard results which will be used in the course of the paper. We begin with the Sobolev-type inequality~\cite[Lemma 2.3]{dingLocalBoundednessHolder2021}.

\begin{theorem}[\cite{dingLocalBoundednessHolder2021}]
    Let $t_2>t_1>0$. Suppose $s\in(0,1)$ and $1\leq p<\infty$. Then for any $f\in L^p(t_1,t_2;W^{s,p}(B_r))\cap L^\infty(t_1,t_2;L^2(B_r))$, we have
    \begin{align}\label{sobolev}
        \int\limits_{t_1}^{t_2}&\fint\limits_{B_r}|f(x,t)|^{p\left(1+\frac{2s}{N}\right)}\,dx\,dt\nonumber\\
        &\leq C\left(r^{sp}\int\limits_{t_1}^{t_2}\int\limits_{B_r}\fint\limits_{B_r}\frac{|f(x,t)-f(y,t)|^p}{|x-y|^{N+sp}}\,dx\,dy\,dt+\int\limits_{t_1}^{t_2}\fint\limits_{B_r}|f(x,t)|^p\,dx\,dt\right)\nonumber\\
        &\times\left(\sup_{t_1<t<t_2}\fint_{B_r}|f(x,t)|^2\,dx\right)^{\frac{sp}{N}},
    \end{align} where $C>0$ depends on $N$, $s$ and $p$. 
\end{theorem}

We shall make use of the following well-known iteration lemma whose proof is similar to \cite[Lemma 6.1]{giustiDirectMethodsCalculus2003}.

\begin{lemma}\label{iterlemma}
	Let $Z(t)$ be a bounded non-negative function in the interval $\rho,R$. Assume that for $\rho\leq t<s\leq R$ we have
	\begin{align*}
		Z(t)\leq [\alpha,\vartheta)[A(s-t)^{-\alpha}+B(s-t)^{-\beta}+C(s-t)^{-\gamma}+D(s-t)^{-\delta}+E]+\vartheta Z(s)
	\end{align*} with $A,B,C,D,E\geq 0$, $\alpha,\beta,\gamma,\delta>0$ and $0\leq \vartheta<1$. Then,
\begin{align*}
	Z(\rho)\leq c(\alpha,\vartheta)[A(R-\rho)^{-\alpha}+B(R-\rho)^{-\beta}+C(R-\rho)^{-\gamma}+D(R-\rho)^{-\delta}+E].
\end{align*}
\end{lemma}

We recall the following well known lemma concerning the geometric convergence of sequence of numbers (see \cite[Lemma 4.1 from Section I]{dibenedettoDegenerateParabolicEquations1993} for the details): 
\begin{lemma}\label{geo_con}
	Let $\{Y_n\}$, $n=0,1,2,\ldots$, be a sequence of positive number, satisfying the recursive inequalities 
	\[ Y_{n+1} \leq C b^n Y_{n}^{1+\alpha},\]
	where $C > 1$, $b>1$, and $\alpha > 0$ are given numbers. If 
	\[ Y_0 \leq  C^{-\frac{1}{\alpha}}b^{-\frac{1}{\alpha^2}},\]
	then $\{Y_n\}$ converges to zero as $n\to \infty$. 
\end{lemma}

Finally, we will make use of the following version of an interpolation lemma from \cite[p. 13, Lemma 4.3]{dibenedettoDegenerateParabolicEquations1993}. 

\begin{lemma}
    \label{interlope_lemma}
    Let $\{Y_n\}$, $n=0,1,2,\ldots$, be a sequence of equibounded positive numbers, satisfying the recursive inequalities \[ Y_{n} \leq C b^n Y_{n+1}^{1-\alpha}+\mathcal{K},\]
	where $C > 1$, $b>1, \mathcal{K}>0$, and $\alpha\in (0,1)$ are given numbers. Then \[ Y_0 \leq \left(\frac{2C}{b^{1-\frac{1}{\alpha}}}\right)^{\frac{1}{\alpha}}+2\mathcal{K}. \]
\end{lemma}

\begin{proof}
    For $\ve\in (0,1)$, we write
    \[ Y_{n} \leq C b^n \ve^{\alpha-1} (\ve Y_{n+1})^{1-\alpha}+\mathcal{K} \] and apply Young's inequality with exponents $\frac{1}{1-\alpha}$ and $\frac{1}{\alpha}$ to get 
    \[ Y_{n} \leq \ve Y_{n+1} + \left( \frac{C}{\ve^{1-\alpha}}\right)^{\frac{1}{\alpha}}b^{\frac{n}{\alpha}} +\mathcal{K}. \]
    By iteration, we receive
    \[ Y_{0} \leq \ve^{n} Y_{n} + \left( \frac{C}{\ve^{1-\alpha}}\right)^{\frac{1}{\alpha}}\sum_{i=0}^{n-1}\left(b^{\frac{1}{\alpha}}\ve\right)^i +\mathcal{K}\sum_{i=0}^{n-1}\ve^i. \]Now, choose $b^{1/\alpha}\ve=\frac{1}{2}$ and take the limit as $n\to\infty$ to conclude the result.
\end{proof}

\section{Caccioppoli inequality}\label{sec3}

In this section, we will prove a Caccioppoli inequality for parabolic minimizers of~\cref{maineq}. The notion of parabolic minimizers was defined in~\cref{defparmin}. In the paper~\cite{prasadExistenceVariationalSolutions2021a}, we have shown that a variational solution is a parabolic minimizer. Hence, there is no loss of generality to work with parabolic minimizers. 

\begin{lemma}
    \label{Caccioppoli}
    Let $u\in C^0(0,T;L^2(\Om))\cap L^p(0,T;W^{s,p}_{(u_0)}(\Om))$ be a variational solution of~\cref{maineq}, where the initial data satisfies~\cref{datahypo}. Further assume that
    \[u\in L^\infty(0,T;L^{p-1}_{sp}(\RR^N))\cap L^\infty(0,T;L^{q-1}_{qs'}(\RR^N)).\] Then for all cylinders $Q_{R,\theta}(z_0)\Subset \Om_T$ and $k,\tau_1,\tau_2$ with $0<r\leq \rho_1<\rho_2\leq R, t_0-\tht<t_0-\tau_2<t_0-\tau_1\leq t_0-\frac{\tht}{2}$ and $q\leq \delta\leq p_*$, we have
    \begin{align}
        \sup\limits_{t_0-\tau_1<t<t_0}\int\limits_{B_{r}(x_0)}|w_+(\cdot,t)|^2\,dx&+\int\limits_{t_0-\tau_1}^{t_0}\int\limits_{B_{r}(x_0)}\int\limits_{B_{r}(x_0)}\frac{H(x,y,w_+(x,t)-w_+(y,t))}{|x-y|^N}\,dx\,dy\,dt\nonumber\\
        &\leq C\frac{R^{p(1-s)}}{(R-r)^p}\norm{w_+}^p_{L^p(Q_{R,\tau_2}(z_0))}+ C\frac{R^{q(1-s')}}{(R-r)^q}\norm{w_+}^q_{L^q(Q_{R,\tau_2}(z_0))}\nonumber\\
        &\quad\quad + C\frac{R^N}{(R-r)^{N+sp}}\norm{w_+}_{L^1(Q_{R,\tau_2}(z_0))}\text{Tail}_{p,s,\infty}(w_+;x_0,r,I)^{p-1}\nonumber\\
        &\quad\quad\quad + C\frac{R^N}{(R-r)^{N+s'q}}\norm{w_+}_{L^1(Q_{R,\tau_2}(z_0))}\text{Tail}_{q,s',\infty}(w_+;x_0,r,I)^{q-1}\nonumber\\
        &\quad\quad\quad\quad + C\int\limits_{t_0-\tau_2}^{t_0}\int\limits_{B_{R}(x_0)} w_+^2(x,t)|\de_t\psi(t)|\,dx\,dt. 
    \end{align}  where $w_+=(u-k)_+$, $\psi \in C^1(t_0-\theta,t_0)$ be a test function such that $\psi = 0$ on $(t_0-\theta,t_0-\tau_2)$, $\psi = 1$ on $(t_0-\tau_1,t_0)$ and $0\leq \psi \leq 1$ and the positive constant $C=C(N,p,q,M,s,s')$.
\end{lemma}

\begin{proof}
    Without loss of generality, we assume that $x_0 = 0$. For clarity of reading we write, 
    \[
    \frac{\,dx\,dy}{|x-y|^N} = dm
    \]
    \[
    H(x,y,f(x,t)-f(y,t)) = H(f(x,t)-f(y,t))
    \]
    and
    \[
    \eta(t)\,dm\,dt = \,dm_t
    \]
    where $0 \leq \eta(t) \leq 1$ is a cutoff function based on the time interval $I = (t_0-\theta,t_0)$. In fact, we choose 
    \[
    \eta(t) = \chi_{\epsilon}(t)\psi^2(t)
    \]
    where 
    \begin{itemize}
        \item $\psi \in C^1(t_0-\theta,t_0)$,
        \item $\psi = 0$ on $(t_0-\theta,t_0-\tau_2)$,
        \item  $\psi = 1$ on $(t_0-\tau_1,t_0)$ and 
        \item $0\leq \psi \leq 1$
    \end{itemize}
    \[
    \chi_{\epsilon}(t) = \left\{
            \begin{array}{lll}
                1 & \mbox{ in } [t_0-\theta,\tau] \\
                1-\frac{1}{\epsilon}(t-\tau) &\text{ on } (\tau,\tau+\epsilon) \\ 
                0 &\text{ otherwise }
            \end{array}
        \right.
    \]
    and $\tau \in (t_0-\tau_1,t_0)$ is arbitrary.

    In the definition of parabolic minimizer~\cref{defparmin}, we test with 
    \[
    \phi = - \eta(t)\xi^{b}(x)w_{+}
    \]
    and invoke the convexity of $H$ in the final variable to compute
    \begin{align*}
    H((u+\phi)(x,t)&-(u+\phi)(y,t)) = \\
        &H((1-\eta)(u(x,t) - u(y,t)) + \eta(u(x,t)-u(y,t)+\phi^*(x,t)    -\phi^*(y,t)))\\
        &\leq (1-\eta)H(u(x,t)-u(y,t)) + \eta H(v(x,t)-v(y,t))
    \end{align*}
    
    where
    \[
    \phi^* = -\xi^{b}(x)w_{+}
    \]
    \[
    v = u + \phi^* = u -\xi^{b}(x)w_{+},
    \]
    and $0 \leq \xi(x) \leq 1$ is a smooth test function supported in $B_{(\rho_1+\rho_2)/2}$ with $\xi = 1$ on $B_{\rho_1}$. 
    Since $u = v$ outside the ball $B_{\rho_2}$, we are led to the following inequality
    \begin{align}\label{best1}
    \int_{\text{supp}\,\phi} u\phi_t \,dz \leq \int\limits_{I}\iint\limits_{C_{B_{\rho_2}}}  H(v(x,t)-v(y,t)) - H(u(x,t)-u(y,t)) \,dm_t 
    \end{align}

    We begin by estimating the right hand side of \cref{best1}. We split the integral as 
    \begin{align*}
        \int\limits_{I}\iint\limits_{C_{B_{\rho_2}}} & = \int\limits_{I}\iint\limits_{B_{\rho_2}^2} +  \int\limits_{I}\iint\limits_{B_{\rho_2} \times (\bb{R}^N\setminus B_{\rho_2})} + \int\limits_{I}\iint\limits_{(\bb{R}^N\setminus B_{\rho_2})\times B_{\rho_2} }\\
    & =  \int\limits_{I}\iint\limits_{B_{\rho_1}^2}\, +\, \int\limits_{I}\iint\limits_{B_{\rho_2}^2 \setminus B_{\rho_1}^2} \, +\,\int\limits_{I}\iint\limits_{B_{\rho_2} \times (\bb{R}^N\setminus B_{\rho_2})} + \int\limits_{I}\iint\limits_{(\bb{R}^N\setminus B_{\rho_2})\times B_{\rho_2} }\\
    & := A_1 + A_2 + A_3 + A_4,
    \end{align*}   with the obvious meanings for $A_1, A_2, A_3$ and $A_4$.

    To estimate the local integrals $A_1$ and $A_2$ we consider the following cases pointwise for $(x,t)$ and $(y,t)$ in $C_{B_{\rho_2}} \times I$. 
    \begin{itemize}
        \item If $(x,t) \notin A_+(k)$ or $(y,t) \notin A_+(k)$ then
        \begin{equation}\label{eq:A}
            H(v(x,t)-v(y,t)) - H(u(x,t)-u(y,t)) \leq 0
        \end{equation}
        \item If $(x,t), (y,t) \in A_+(k,\rho_1,I)$ then
        \begin{equation}\label{eq:B}
            H(v(x,t)-v(y,t)) - H(u(x,t)-u(y,t)) = - H(w_+(x,t)-w_+(y,t))
        \end{equation}
        \item If $(x,t) \in A_+(k,\rho_1,I)$ and $ (y,t) \notin A_+(k)$ then
        \begin{align}\label{eq:C}
        H(v(x,t)-v(y,t)) - H(u(x,t)-u(y,t)) \leq - \frac{1}{2}& H(w_+(x,t)-w_+(y,t)) - \frac{p}{2}w_+(x,t)\frac{w_{-}^{p-1}(y,t)}{|x-y|^{sp}} \nonumber\\   &-a(x,y)\frac{q}{2}(y,t)w_+(x,t)\frac{w_{-}^{q-1}(y,t)}{|x-y|^{s'q}}
        \end{align}
        \item If $(x,t),(y,t) \in A_+(k)$ then
        \begin{align}\label{eq:D}
        H(v(x,t)-v(y,t)) \leq H(w_+(x,t)-w_+(y,t)) + b^p\max\{w_+(x,t),\,w_+(y,t)\}^p\frac{|\xi(x) -\xi(y)|^p}{|x-y|^{sp}}\nonumber \\   +M\,b^q\max\{w_+(x,t),\,w_+(y,t)\}^q\frac{|\xi(x) -\xi(y)|^q}{|x-y|^{s'q}}
        \end{align}
    \end{itemize}
        
    The estimates in $\eqref{eq:A}$ and $\eqref{eq:B}$ follow directly from the definition of $v$. For the third item we note that if $(x,t) \in A_+(k,\rho_1,I)$ and $ (y,t) \notin A_+(k)$ then
    \[
    |v(x,t) - v(y,t)| = w_{-}(y,t)
    \]
    and
    \[
    |u(x,t) - u(y,t)| = w_{+}(x,t) + w_{-}(y,t).
    \]
    Invoking the following inequality (\cite[Lemma 4.1]{cozziRegularityResultsHarnack2017}) for $a,b \geq 0$ and any $r \geq 1$ 
    \[
    a^r - (a+b)^r \leq  - \frac{1}{2}b^r -\frac{1}{2}ra^{r-1}b
    \]
    and recalling that $w_+(y,t) = 0$ for $(y,t) \notin A_{+}(k)$ we get $\eqref{eq:C}$. Finally, let $(x,t),(y,t) \in A_+(k)$. We can assume without loss of generality that $\xi(x) \geq \xi(y)$ since the estimate is symmetric in $x$ and $y$; further, if $\xi(x) = 0$ then $\xi(y) = 0$ and the estimate is clearly true, so we can also assume that $\xi(x)>0$. In particular, the convexity of $X \mapsto X^b$ implies that
    \[
    |\xi^b(x) - \xi^b(y)| \leq b\xi^{b-1}(x)|\xi(x) - \xi(y)| 
    \]
    Next, we note that
    \begin{align*}
    |v(x,t) - v(y,t)| = |(1-\xi^b(x))(w_+(x,t)-w_+(y,t))+(\xi^b(y)-\xi^b(x))w_+(y,t)|    
    \end{align*}
    Since $H$ is convex in the final variable, we get
    \begin{align*}
        H(v(x,t) - v(y,t)) 
        &\leq H\left((1-\xi^b(x))|w_+(x,t)-w_+(y,t)|+\xi^b(x)\frac{bw_+(y,t) |\xi(y)-\xi(x)|}{\xi(x)}\right)
        \\
        &\leq (1-\xi^b(x))H\left(|w_+(x,t)-w_+(y,t)|\right)+\xi^b(x)H\left(\frac{bw_+(y,t) |\xi(y)-\xi(x)|}{\xi(x)}\right)
        \\
        &\leq H\left(|w_+(x,t)-w_+(y,t)|\right) + \frac{\xi^b(x)}{\xi^p(x)}b^pw^p_+(y,t)\frac{|\xi(x) -\xi(y)|^p}{|x-y|^{sp}}\\
        &\qquad+ M\frac{\xi^b(x)}{\xi^q(x)}b^qw^q_+(y,t)\frac{|\xi(x) -\xi(y)|^q}{|x-y|^{s'q}}
    \end{align*}
    Since the inequality is symmetric in $x,y$ we get \eqref{eq:D}. It is here that we fix the choice of $b$ to be $p+q$.  
    
    We now estimate the local integrals $A_1$ and $A_2$. For the first local integral $A_1$, we use \eqref{eq:A},\eqref{eq:B} and \eqref{eq:C} to get
    \begin{align*}
        A_1=\int\limits_{I}\iint\limits_{B_{\rho_1}^2} & H(v(x,t)-v(y,t)) - H(u(x,t)-u(y,t)) \,dm_t\\
        &\leq   -\frac{1}{2}\int\limits_{I}\iint\limits_{B_{\rho_1}^2}H(w_+(x,t)-w_+(y,t))\,dm_t - \frac{p}{2} \int\limits_{I}\iint\limits_{B_{\rho_1}^2}w_+(x,t)\frac{w_{-}^{p-1}(y,t)}{|x-y|^{sp}} \,dm_t \\   &\qquad -a(x,y)\frac{q}{2}\int\limits_{I}\iint\limits_{B_{\rho_1}^2}w_+(x,t)\frac{w_{-}^{q-1}(y,t)}{|x-y|^{s'q}} \,dm_t
    \end{align*}
    For the second local integral $A_2$, we use \eqref{eq:A}, \eqref{eq:C} and \eqref{eq:D}  to get
    \begin{align*}
        A_2=\int\limits_{I}\iint\limits_{B_{\rho_2}^2 \setminus B_{\rho_1}^2} & H(v(x,t)-v(y,t)) - H(u(x,t)-u(y,t)) \,dm_t\\ 
        &\leq \int\limits_{I}\iint\limits_{B_{\rho_2}^2 \setminus B_{\rho_1}^2}H(w_+(x,t)-w_+(y,t)) \, dm_t \\
        &\qquad+ \int\limits_{I}\iint\limits_{B_{\rho_2}^2 \setminus B_{\rho_1}^2} b^p\max\{w_+(x,t),\,w_+(y,t)\}^p\frac{|\xi(x) -\xi(y)|^p}{|x-y|^{sp}} \,dm_t \\  &\qquad\qquad+\int\limits_{I}\iint\limits_{B_{\rho_2}^2 \setminus B_{\rho_1}^2}M\,b^q\max\{w_+(x,t),\,w_+(y,t)\}^q\frac{|\xi(x) -\xi(y)|^q}{|x-y|^{s'q}}\,dm_t \\
        &\qquad\qquad\qquad- \frac{p}{2}\int\limits_{I}\int\limits_{B_{\rho_1}}w_+(x,t)\int\limits_{B_{\rho_2}\setminus B_{\rho_1}
        }\frac{w_{-}^{p-1}(y,t)}{|x-y|^{sp}}  \, dm_t \\  &\qquad\qquad\qquad\qquad-\frac{q}{2}\int\limits_{I}\int\limits_{B_{\rho_1}}w_+(x,t)\int\limits_{B_{\rho_2}\setminus B_{\rho_1}
        }a(x,y)\frac{w_{-}^{q-1}(y,t)}{|x-y|^{s'q}}\,dm_t
    \end{align*}
    Combining the estimates for $A_1$ and $A_2$, we get
    \begin{align*}
       A_1+A_2=\int\limits_{I}\iint\limits_{B_{\rho_2}^2} & H(v(x,t)-v(y,t)) - H(u(x,t)-u(y,t)) \,dm_t \\ 
       &\leq  -\frac{1}{2}\int\limits_{I}\iint\limits_{B_{\rho_1}^2}H(w_+(x,t)-w_+(y,t))\,dm_t \\
       &\qquad+  \int\limits_{I}\iint\limits_{B_{\rho_2}^2 \setminus B_{\rho_1}^2}H(w_+(x,t)-w_+(y,t)) \, dm_t \\
       &\qquad\qquad+ \int\limits_{I}\iint\limits_{B_{\rho_2}^2 \setminus B_{\rho_1}^2} b^p\max\{w_+(x,t),\,w_+(y,t)\}^p\frac{|\xi(x) -\xi(y)|^p}{|x-y|^{sp}} \,dm_t \\   &\qquad\qquad\qquad+\int\limits_{I}\iint\limits_{B_{\rho_2}^2 \setminus B_{\rho_1}^2}Mb^q\max\{w_+(x,t),\,w_+(y,t)\}^q\frac{|\xi(x) -\xi(y)|^q}{|x-y|^{s'q}}\,dm_t,
    \end{align*} where we have dropped two negative terms.
    
    Next, we estimate the nonlocal integral $A_3$. We have that 
    \begin{itemize}
        \item if $(x,t),(y,t) \notin A_+(k,(\rho_1+\rho_2)/2,I)$ then \begin{equation}\label{eq:E}
            |v(x,t)-v(y,t)| = |u(x,t)-u(y,t)|
        \end{equation}
        \item if $(x,t) \in A_+(k)$ and $y \in A_+(k) \setminus B_{\rho_2}$ then
        \begin{equation}\label{eq:F}
            H(v(x,t)-v(y,t)) - H(u(x,t)-u(y,t)) 
            \leq pw_+(x,t)\frac{w_{+}^{p-1}(y,t)}{|x-y|^{sp}} + qMw_+(x,t)\frac{w_{+}^{q-1}(y,t)}{|x-y|^{s'q}} 
        \end{equation}
        where we recall that $M \geq a(x,y) \geq 0$.  
    \end{itemize}
    To see that \eqref{eq:F} holds we write for $r \in \{p,q\}$ and $f \in \{1,a(x,y)\}$
    \[
    |v(x,t)-v(y,t)|^rf - |u(x,t)-u(y,t)|^rf  = |(1-\xi^b(x))w_+(x,t)-w_+(y,t)|^rf - |w_+(x,t)-w_+(y,t)|^rf 
    \]
    and employ the following inequality (\cite[Lemma 4.2]{cozziRegularityResultsHarnack2017}) \[
    |\delta a - b|^r - |a-b|^r \leq rb^{r-1}a
    \]   
    which holds whenever $r\geq 1$, $a,b \geq 0$ and $0 \leq \delta \leq 1$. So using \eqref{eq:A},\eqref{eq:C},\eqref{eq:E} and \eqref{eq:F} we get
    \begin{align*}
        A_3= \iint\limits_{B_{\rho_2} \times (\bb{R}^N\setminus B_{\rho_2})} & H(v(x,t)-v(y,t)) - H(u(x,t)-u(y,t)) \,dm_t
        \\
        &\leq \int\limits_{I} \int\limits_{B_{\frac{{\rho_1+\rho_2}}{2}}}\int\limits_{\bb{R}^N\setminus B_{\rho_2}}p w_+(x,t)\frac{w_{+}^{p-1}(y,t)}{|x-y|^{sp}} + qMw_+(x,t)\frac{w_{+}^{q-1}(y,t)}{|x-y|^{s'q}} \,dm_t
        \\
        &\qquad- \int\limits_{I} \int\limits_{B_{\rho_1}}\int\limits_{\bb{R}^N \setminus B_{\rho_2}}\frac{p}{2}w_+(x,t)\frac{w_{-}^{p-1}(y,t)}{|x-y|^{sp}}+a(x,y)\frac{q}{2}(y,t)w_+(x,t)\frac{w_{-}^{q-1}(y,t)}{|x-y|^{s'q}} \,dm_t
    \end{align*}
    where we are integrating first in the variable $y$ and then in the variable $x$. The estimate for $A_4$ is similar. Combining the estimates for $A_1, A_2, A_3$ and $A_4$ with \cref{best1}, we get   
    \begin{align}
       \int_{\text{supp}\,\phi} & u\phi_t \,dz + \frac{3}{2}\int\limits_{I}\iint\limits_{B_{\rho_1}^2}H(w_+(x,t)-w_+(y,t))\,dm_t\nonumber\\
       &\leq  \int\limits_{I}\iint\limits_{B_{\rho_2}^2}H(w_+(x,t)-w_+(y,t))\,dm_t \nonumber\\
       &\quad+ C\underbrace{\int\limits_{I}\iint\limits_{B_{\rho_2}^2 \setminus B_{\rho_1}^2} \max\{w_+(x,t),\,w_+(y,t)\}^p\frac{|\xi(x) -\xi(y)|^p}{|x-y|^{sp}} \,dm_t}_{Z_1} \nonumber\\   
       &\quad\quad+MC\underbrace{\int\limits_{I}\iint\limits_{B_{\rho_2}^2 \setminus B_{\rho_1}^2}\max\{w_+(x,t),\,w_+(y,t)\}^q\frac{|\xi(x) -\xi(y)|^q}{|x-y|^{s'q}}\,dm_t}_{Z_2}\nonumber\\
       &\quad\quad\quad+ C\underbrace{\int\limits_{I} \int\limits_{B_{\frac{{\rho_1+\rho_2}}{2}}}\int\limits_{\bb{R}^N\setminus B_{\rho_2}} w_+(x,t)\frac{w_{+}^{p-1}(y,t)}{|x-y|^{sp}}}_{Z_3}\nonumber \\
       &\quad\quad\quad\quad + MC\underbrace{\int\limits_{I} \int\limits_{B_{\frac{{\rho_1+\rho_2}}{2}}}\int\limits_{\bb{R}^N\setminus B_{\rho_2}}w_+(x,t)\frac{w_{+}^{q-1}(y,t)}{|x-y|^{s'q}} \,dm_t}_{Z_4}\label{best2}
    \end{align}
    where $C > 0$ is a constant depending only on $p,q$. We have once again dropped the negative terms coming from $A_3$ and $A_4$. 
    
    \paragraph{Estimate for $Z_3$ and $Z_4$}
    If $x \in B_{(\rho_1+\rho_2)/2}$ and $y\notin B_{\rho_2}$ then
    \[
    |x-y| \geq |y| - |x| \geq |y| - \frac{\rho_1+\rho_2}{2\rho_2}|y| \geq \frac{\rho_2-\rho_1}{2R}|y|
    \]
    and so we may estimate the nonlocal integrals for $g \in \{p,q\}$ and $h \in \{s,s'\}$ as
    \begin{align}
    \int\limits_{I} \eta^2(t) \int\limits_{B_{\frac{{\rho_1+\rho_2}}{2}}} & w_+(x,t)\int\limits_{\bb{R}^N\setminus B_{\rho_2}}  \frac{w_{+}^{g-1}(y,t)}{|x-y|^{gh+N}} \, dy \, dx \, dt\nonumber\\
    &\leq C\left(\frac{R}{\rho_2-\rho_1}\right)^{N+gh}\int\limits_{I} \eta^2(t) \int\limits_{B_{\frac{{\rho_1+\rho_2}}{2}}}w_+(x,t)\int\limits_{\bb{R}^N\setminus B_{\rho_2}} \frac{w_{+}^{g-1}(y,t)}{|y|^{gh+N}} \, dy \, dx \, dt \nonumber\\
    &\qquad\leq C\left(\frac{R}{\rho_2-\rho_1}\right)^{N+gh}\norm{w_+}_{L^1(Q_{R,\tau_2})}\text{Tail}_{g,h,\infty}(w_+;0,r,I)^{g-1}R^{-gh}\label{best3}
    \end{align}

    \paragraph{Estimate for $Z_1$ and $Z_2$}
    We also have 
    \begin{align}
        \int\limits_{I}\iint\limits_{B_{\rho_2}^2 \setminus B_{\rho_1}^2}\max\{w_+(x,t),\,w_+(y,t)\}^g & \frac{|\xi(x) -\xi(y)|^g}{|x-y|^{gh}}\,dm_t\nonumber \\
        &\leq \int\limits_{I}\iint\limits_{B_{R}^2 }\max\{w_+(x,t),\,w_+(y,t)\}^g\frac{|\xi(x) -\xi(y)|^g}{|x-y|^{gh}}\,dm_t\nonumber
        \\
        &\qquad\leq C\frac{R^{g(1-h)}}{(\rho_2-\rho_1)^g}\norm{w_+}^g_{L^g(Q_{R,\tau_2})}\label{best4}
    \end{align}
    The constants $C$ above now also depend on the dimension.
    
    Substituting \cref{best3} and \cref{best4} in \cref{best2}, we get
    \begin{align}
        \int_{\text{supp}\in \phi} u\phi_t \,dz + \frac{3}{2}\int\limits_{I}\iint\limits_{B_{\rho_1}^2} & H(w_+(x,t)-w_+(y,t))\,dm_t \leq \int\limits_{I}\iint\limits_{B_{\rho_2}^2}H(w_+(x,t)-w_+(y,t))\,dm_t \nonumber\\ 
        &+C\frac{R^{p(1-s)}}{(\rho_2-\rho_1)^p}\norm{w_+}^p_{L^p(Q_{R,\tau_2})}+ MC\frac{R^{q(1-s')}}{(\rho_2-\rho_1)^q}\norm{w_+}^q_{L^q(Q_{R,\tau_2})}\nonumber\\
        &\qquad+C\left(\frac{R}{\rho_2-\rho_1}\right)^{N+sp}\norm{w_+}_{L^1(Q_{R,\tau_2})}\text{Tail}_{p,s,\infty}(w_+;0,r,I)^{p-1}R^{-sp}\nonumber\\
        &\qquad\qquad+ MC\left(\frac{R}{\rho_2-\rho_1}\right)^{N+s'q}\norm{w_+}_{L^1(Q_{R,\tau_2})}\text{Tail}_{q,s',\infty}(w_+;0,r,I)^{q-1}R^{-s'q}\label{best5}
    \end{align}
    where $C = C(p,q,N) > 0$. 
    
    We compute the first term on the left hand side of \cref{best5}
    \begin{align}
        \int_{\text{supp}\in \phi} u\phi_t \,dz &=-\int_0^T \int_\Om \partial_t u \, \phi \,dz\nonumber\\
        &=\int_0^T\int_\Om \partial_t u\,\chi_{\ve}(t)\psi^2(t)\xi^b(x)(u-k)_+\,dz\nonumber\\
        &= \int_0^T\int_{\Omega} \frac{1}{2}\de_t (w_+)^2 \chi_{\epsilon}(t)\psi^2(t)\xi^b(x) \,dz \nonumber\\
        &= - \frac{1}{2}\int_0^T\int_{\Omega}(w_+)^2\left(\chi_{\epsilon}'\psi^2\xi^b + 2\chi_{\epsilon}\psi\psi'\xi^b\right)\,dz\nonumber\\
        & \xrightarrow{\epsilon \rightarrow 0+} \frac{1}{2}\int_{\Omega}|w_+(\cdot,\tau)|^2 \xi^b\,dx - \int_0^T\chi_{(t_0-\theta,\tau)}\psi\int_{\Omega} w_+^2 |\de_t \psi| \xi^b \, dz \label{best6}
    \end{align}
    Since $\tau \in (t_0-\tau_1,t_0)$ and $\xi = 1$ on $B_{\rho_1}$ we combine \cref{best6} and \cref{best5} to get
    \begin{align}
        \sup\limits_{t_0-\tau_1<t<t_0}&\int\limits_{B_{\rho_1}}|w_+(\cdot,t)|^2\,dx + \frac{3}{2}\int\limits_{I}\iint\limits_{B_{\rho_1}^2} H(w_+(x,t)-w_+(y,t))\,dm_t \nonumber\\
        &\leq \int\limits_{I}\iint\limits_{B_{\rho_2}^2}H(w_+(x,t)-w_+(y,t))\,dm_t\nonumber\\
        &\quad+C\frac{R^{p(1-s)}}{(\rho_2-\rho_1)^p}\norm{w_+}^p_{L^p(Q_{R,\tau_2})}+ MC\frac{R^{q(1-s')}}{(\rho_2-\rho_1)^q}\norm{w_+}^q_{L^q(Q_{R,\tau_2})}\nonumber\\
        &\quad\quad+ C\left(\frac{R}{\rho_2-\rho_1}\right)^{N+sp}\norm{w_+}_{L^1(Q_{R,\tau_2})}\text{Tail}_{p,s,\infty}(w_+;0,r,I)^{p-1}R^{-sp}\nonumber\\
        &\quad\quad+MC\left(\frac{R}{\rho_2-\rho_1}\right)^{N+s'q}\norm{w_+}_{L^1(Q_{R,\tau_2})}\text{Tail}_{q,s',\infty}(w_+;0,r,I)^{q-1}R^{-s'q}\nonumber \\
        &\quad\quad\quad+ \int\limits_{t_0-\tau_2}^{t_0}\int\limits_{B_{\rho_2}} w_+^2(x,t)|\de_t\psi(t)|\,dx\,dt\label{best7}
    \end{align}
    where, due to the passage of limit $\epsilon \rightarrow 0+$, $dm_t$ becomes
    \[
    dm_t = \psi^2(t)\chi_{(t_0-\theta,t_0)}(t) dm.
    \]

    The convergence on the left hand side is due to Fatou's lemma and the convergence on the right hand side is due to dominated convergence theorem.

    Now \cref{iterlemma} implies that we can absorb the first integral on the right hand side of \cref{best7} into the spacetime integral on the left hand side of \cref{best7} to get
    \begin{align*}
        \sup\limits_{t_0-\tau_1<t<t_0}\int\limits_{B_{r}}|w_+(\cdot,t)|^2\,dx & + \int\limits_{I}\iint\limits_{B_{r}^2} H(w_+(x,t)-w_+(y,t))\,dm_t \\
        &\lesssim \frac{R^{p(1-s)}}{(R-r)^p}\norm{w_+}^p_{L^p(Q_{R,\tau_2})}+ \frac{R^{q(1-s')}}{(R-r)^q}\norm{w_+}^q_{L^q(Q_{R,\tau_2})}\\
        &\quad+ \left(\frac{R}{R-r}\right)^{N+sp}\norm{w_+}_{L^1(Q_{R,\tau_2})}\text{Tail}_{p,s,\infty}(w_+;0,r,I)^{p-1}R^{-sp}\\
        &\quad\quad+\left(\frac{R}{R-r}\right)^{N+s'q}\norm{w_+}_{L^1(Q_{R,\tau_2})}\text{Tail}_{q,s',\infty}(w_+;0,r,I)^{q-1}R^{-s'q} \\
        &\quad\quad\quad+ \int\limits_{t_0-\tau_2}^{t_0}\int\limits_{B_{R}} w_+^2(x,t)|\de_t\psi(t)|\,dx\,dt
    \end{align*}    
    Finally, recalling what $dm_t$ stands for we get    
    \begin{align*}
            \sup\limits_{t_0-\tau_1<t<t_0}&\int\limits_{B_{r}}|w_+(\cdot,t)|^2\,dx+\int\limits_{t_0-\tau_1}^{t_0}\int\limits_{B_{r}}\int\limits_{B_{r}}\frac{H(x,y,w_+(x,t)-w_+(y,t))}{|x-y|^N}\,dx\,dy\,dt\nonumber\\
            &\lesssim \frac{R^{p(1-s)}}{(R-r)^p}\norm{w_+}^p_{L^p(Q_{R,\tau_2})}\\
            &\qquad + \frac{R^{q(1-s')}}{(R-r)^q}\norm{w_+}^q_{L^q(Q_{R,\tau_2})}\\
            &\qquad + \left(\frac{R}{R-r}\right)^{N+sp}\norm{w_+}_{L^1(Q_{R,\tau_2})}\text{Tail}_{p,s,\infty}(w_+;0,r,I)^{p-1}R^{-sp}\\
            &\qquad + \left(\frac{R}{R-r}\right)^{N+s'q}\norm{w_+}_{L^1(Q_{R,\tau_2})}\text{Tail}_{q,s',\infty}(w_+;0,r,I)^{q-1}R^{-s'q}\\
            &\qquad + \int\limits_{t_0-\tau_2}^{t_0}\int\limits_{B_{R}} w_+^2(x,t)|\de_t\psi(t)|\,dx\,dt. 
        \end{align*}    
        
    The conclusion follows.     
    \end{proof}
    \begin{remark}
        In our proof, we neglected many negative terms on the right hand side. If we were to include these negative terms we would get a tighter energy bound with positive terms of the form
        \[
        \frac{p}{2} \int\limits_{t_0-\tau_1}^{t_0}\int\limits_{B_{\rho_1}}w_+(x,t)\int\limits_{B_{\rho_1}}\frac{w_{-}^{p-1}(y,t)}{|x-y|^{N+sp}} \,dx\,dy\,dt + \frac{q}{2}\int\limits_{t_0-\tau_1}^{t_0}\int\limits_{B_{\rho_1}}w_+(x,t)\int\limits_{B_{\rho_1}}a(x,y)\frac{w_{-}^{q-1}(y,t)}{|x-y|^{N+s'q}} \,dx\,dy\,dt
        \]
        on the left hand side. However, the version we prove is sufficient for our purposes.
        \end{remark}
        \begin{remark}
        The finiteness of the tail term is an additional assumption that we must make. However, such additional assumptions must be imposed even in the case $a=0$ i.e. parabolic fractional $p-$Laplacian as done in  \cite{stromqvistLocalBoundednessSolutions2019} and \cite{dingLocalBoundednessHolder2021}, unlike in the elliptic case where finiteness could be guaranteed if the corresponding tail term for the data was finite as done in \cite{byunOlderRegularityWeak2021}.
        \end{remark}

\section{De Giorgi iteration}\label{sec4}

In this section, we will prove the boundedness estimate in the case $q<p_*$. 

\subsection{Recursive estimates}
We begin by proving some recursive estimates which will be used in the proof of the boundedness estimate. We set up the notation as follows. Let $(x_0,t_0)\in \Om_T$, $R>0$, and $Q_{R,\tht}=B_R(x_0)\times (t_0-\tht,t_0)$, such that $\overline{B_R(x_0)}\subseteq \Om$ and $[t_0-\tht,t_0]\subseteq (0,T)$. 
For $\sigma \in [1/2,1)$, we define the following sequences
\begin{align}\label{seqR}
    R_j=\sigma R+(1-\sigma) R 2^{-j}\mbox{ and } \tht_j=\sigma\tht+(1-\sigma)\tht {2^{-j}},\, j=0,1,\ldots
\end{align} so that $R_0=R$, $R_\infty=\sigma R $, $\tht_0=\tht$ and $\tht_\infty=\sigma \tht$. Corresponding to these sequences, we define a sequence of nested cylinders
\begin{align}\label{seqQ}
    &Q_j(z_0)=Q_{R_j,\tht_j}(z_0)=B_j\times I_j = B_{R_j}(x_0)\times [t_0-\tht_j,t_0]\nonumber\\
    &Q_0(z_0)=Q_{R,\tht},\,Q_{\infty}=Q_{\sigma R,\sigma \tht}.
\end{align}
We also denote $\tilde{R_j}:=\frac{R_j+R_{j+1}}{2}$ and $\tilde{\tht_j}:=\frac{\tht_j+\tht_{j+1}}{2}$ so that we have the inclusions \[Q_{j+1}\subset \tilde{Q_j}:=Q_{\tilde{R_j},\tilde{\tht_j}}\subset Q_j.\]

For 
\begin{align}\label{choiceofk}
    \tk\geq \max\left\{\frac{\text{Tail}_{p,s,\infty}(u_+;x_0,\sigma R,I_0)}{2},\frac{\text{Tail}_{q,s',\infty}(u_+;x_0,\sigma R,I_0)}{2}\right\}
\end{align}
to be chosen later, we define a sequence of increasing levels as
\begin{align}\label{seqW}
    k_j=(1-2^{-j})\tk,\,\tk_j=\frac{k_j+k_{j+1}}{2},\,j=0,1,2,\ldots\nonumber\\
    w_j=(u-k_j)_+,\,\tilde{w}_j=(u-\tk_j)_+,\,j=0,1,2,\ldots
\end{align}

For $j=0,1,\ldots$, we choose the cut off functions $\eta_j\in C_0^\infty(I_j)$ satisfying
\begin{align}\label{seqtest}
    0\leq \eta_j\leq 1,\,|\de_t\eta_j|\leq C\frac{2^{spj}}{(1-\sigma)^{sp}\tht}\mbox{ in }\tilde{I}_j,\,\eta_j\equiv 1\mbox{ in }I_{j+1}.
\end{align}

\begin{lemma}\label{recurselemma}
    Let $p>1$ and $u$ be a variational solution to~\cref{maineq} in the sense of~\cref{defvar}. Further assume that
    \[u\in L^\infty(0,T;L^{p-1}_{sp}(\RR^N))\cap L^\infty(0,T;L^{q-1}_{qs'}(\RR^N)).\] Let $(x_0,t_0)\in \Om_T$, $R>0$, and $Q_{R,\tht}=B_R(x_0)\times (t_0-\tht,t_0)$, such that $\overline{B_R(x_0)}\subseteq \Om$ and $[t_0-\tht,t_0]\subseteq (0,T)$. Assume that $\delta>0$ is a number such that $\delta\geq\max\{p, q, 2\}$. Let $B_j,\,\tilde{B}_j,\,I_j,\tilde{I_j}$ be as defined in~\cref{seqQ} and~\cref{seqW}. Then we have for all $j\in\NN$
    \begin{align}\label{recest1}
        \sup\limits_{t\in I_{j+1}}&\fint\limits_{B_{j+1}}|\tilde{w}_j(\cdot,t)|^2\,dx+\int\limits_{I_{j+1}}\int\limits_{B_{j+1}}\fint\limits_{B_{j+1}}\frac{H(x,y,\tw_j(x,t)-\tw_j(y,t))}{|x-y|^N}\,dx\,dy\,dt\nonumber\\
        &\leq C\left[\frac{1}{\sigma^{ps}(1-\sigma)^{N+ps}}+\frac{1}{\sigma^{qs'}(1-\sigma)^{N+qs'}}+\frac{1}{(1-\sigma)^{q}}\right]\times\\
        &\qquad\qquad\left[\frac{2^{(N+sp+\delta-1)j}}{R^{sp}\tk^{\delta-p}}+\frac{2^{(N+qs'+\delta-1)j}}{R^{qs'}\tk^{\delta-q}}+\frac{2^{(sp+\delta-2)j}}{\tht\tk^{\delta-2}}\right]\,\int\limits_{I_j}\fint\limits_{B_j}\tw_j^\delta\,dx\,dt,
    \end{align}
\end{lemma}

\begin{proof}
    Observe that for any $\tau\in [0,q)$, we have
    \begin{align*}
        (u-k_j)_+^{q}&\geq (u-k_j)_+^q \chi_{\{u\geq \tk_j\}}(x,t)\\
        &\geq (\tk_j-k_j)^{q-\tau}(u-k_j)_+^\tau \chi_{\{u\geq \tk_j\}}(x,t)\\
        &\geq C\,\tk^{q-\tau}\,2^{-(q-\tau)j}(u-k_j)_+^\tau \chi_{\{u\geq \tk_j\}}(x,t)\\
        &\geq C\,\tk^{q-\tau}\,2^{-(q-\tau)j}(u-k_j)_+^\tau\mbox{ in }Q_T.
    \end{align*}
    Hence, we may write
    \begin{align}\label{improvK}
        \tilde{w}^\tau_{j}(x,t)\leq C \frac{2^{(q-\tau)j}}{\tk^{q-\tau}}\,w^q_j(x,t)\mbox{ in }Q_T.
    \end{align}
    Now, in~\cref{Caccioppoli}, we take $R=R_j,\,r=R_{j+1}$, $\tau_2=\tht_i$ and $\tau_1=\tht_{i+1}$ to get the following inequality.
    \begin{align}
        \sup\limits_{t\in I_{j+1}}\fint\limits_{B_{j+1}}|\tilde{w}_j(\cdot,t)|^2\,dx&+\int\limits_{I_{j+1}}\int\limits_{B_{j+1}}\fint\limits_{B_{j+1}}\frac{H(x,y,\tw_j(x,t)-\tw_j(y,t))}{|x-y|^N}\,dx\,dy\,dt\nonumber\\
        &\leq C\underbrace{\frac{R_j^{(1-s)p}}{(R_j-R_{j+1})^p}\int\limits_{I_j}\fint\limits_{B_j}{|\tw_j(x,t)|^p}\,dx\,dt}_I\nonumber\\
        &\qquad + C\underbrace{\frac{R_j^{(1-s')q}}{(R_j-R_{j+1})^q}\int\limits_{I_j}\fint\limits_{B_j}{|\tw_j(x,t)|^q}\,dx\,dt}_{II}\nonumber\\
        &\qquad + C\underbrace{\frac{R_j^N}{(R_j-R_{j+1})^{N+sp}}\left(\text{Tail}_{p,s,\infty}(u_+;x_0,R/2,I_0)\right)^{p-1}\, \int\limits_{I_j}\fint\limits_{B_j} \tw_j(x,t)\,dx\,dt}_{III}\nonumber\\
        &\qquad + C\underbrace{\frac{R_j^N}{(R_j-R_{j+1})^{N+s'q}}\left(\text{Tail}_{q,s',\infty}(u_+;x_0,R/2,I_0)\right)^{q-1}\, \int\limits_{I_j}\fint\limits_{B_j} \tw_j(x,t)\,dx\,dt}_{IV}\nonumber\\
        &\qquad + C \underbrace{\int\limits_{I_j}\fint\limits_{B_j} \tw_j^2(x,t)|\de_t\eta(t)|\,dx\,dt}_{V}.
    \end{align}
    Now, we will proceed with estimating each of the terms on the right hand side (subsequently RHS). 

    \paragraph{Estimate for $I$ and $II$} Since the estimate for $I$ and $II$ is similar, we show for $l=p\mbox{ or }q$ and respectively $s_l=s\mbox{ or }s'$ and respectively $I_l=I\mbox{ or }II$ that
    \begin{align}\label{estI}
        I_l \leq C \frac{2^{l(j+1)}}{(1-\sigma)^{l}R^{ls_l}}\int\limits_{I_j}\fint\limits_{B_j}\tw_j^l\,dx\,dt\leq C\frac{2^{\delta(j+1)}}{(1-\sigma)^{l}\tk^{\delta-l}R^{ls_l}}\int\limits_{I_j}\fint\limits_{B_j}\tw_j^\delta\,dx\,dt,
    \end{align} where, in the first inequality, we have used~\cref{seqR}; and, in the last inequality, we have used~\cref{improvK}. 

    \paragraph{Estimate for $III$ and $IV$}

    Once again, since the estimate for $III$ and $IV$ is similar, we will provide calculations for $l=p\mbox{ or }q$ and respectively $s_l=s\mbox{ or }s'$ and respectively $A=III\mbox{ or }IV$.
    
    Notice that
    \begin{align}
        \label{est1}
        \int_{I_j}\fint_{B_j} |\tw_j(x,t)|\,dx\,dt\leq C\frac{2^{(\delta-1)j}}{\tk^{\delta-1}}\int_{I_j}\fint_{B_j} |\tw_j(x,t)|^\delta\,dx\,dt.
    \end{align}



    Using \cref{est1}, \cref{choiceofk}, and \cref{seqR} to estimate $A$, we obtain

    \begin{align}
        \label{estIII}
        A \leq C \frac{2^{(N+ls_l+\delta-1)j}}{\sigma^{ls_l}(1-\sigma)^{N+ls_l}R^{ls_l}\tk^{\delta-l}}\int_{I_j}\fint_{B_j} |\tw_j(x,t)|^\delta\,dx\,dt.
    \end{align}

    \paragraph{Estimate for $V$}
    \begin{align}\label{estV}
        V & \leq C \frac{2^{spj}}{(1-\sigma)^{sp}\tht}\int\limits_{I_j}\fint\limits_{B_j}\tw_j^2\,dx\,dt \leq C\frac{2^{(sp+\delta-2)j}}{(1-\sigma)^{sp}\tht\tk^{\delta-2}}\int\limits_{I_j}\fint\limits_{B_j}\tw_j^\delta\,dx\,dt,
    \end{align} where, in the first inequality, we have used~\cref{seqtest}; and, in the last inequality, we have used~\cref{improvK}.

    Combining the estimates \cref{estI}, \cref{estIII}, and \cref{estV}, we obtain 
    \begin{align}
        \sup\limits_{t\in I_{j+1}}&\fint\limits_{B_{j+1}}|\tilde{w}_j(\cdot,t)|^2\,dx+\int\limits_{I_{j+1}}\int\limits_{B_{j+1}}\fint\limits_{B_{j+1}}\frac{H(x,y,\tw_j(x,t)-\tw_j(y,t))}{|x-y|^N}\,dx\,dy\,dt\nonumber\\
        &\leq C \left[\frac{1}{\sigma^{ps}(1-\sigma)^{N+ps}}+\frac{1}{\sigma^{qs'}(1-\sigma)^{N+qs'}}+\frac{1}{(1-\sigma)^{q}}\right]\times\nonumber\\
        &\qquad\qquad\left[\frac{2^{(N+sp+\delta-1)j}}{R^{sp}\tk^{\delta-p}}+\frac{2^{(N+qs'+\delta-1)j}}{R^{qs'}\tk^{\delta-q}}+\frac{2^{(sp+\delta-2)j}}{\tht\tk^{\delta-2}}\right]\,\int\limits_{I_j}\fint\limits_{B_j}\tw_j^\delta\,dx\,dt
    \end{align} which is \cref{recest1}.
\end{proof}

\begin{lemma}
    Let $p>\frac{2N}{N+2s}$ and $u$ be a variational solution to \cref{maineq}. Further assume that
    \[u\in L^\infty(0,T;L^{p-1}_{sp}(\RR^N))\cap L^\infty(0,T;L^{q-1}_{qs'}(\RR^N)).\] Let $(x_0,t_0)\in \Om_T$, $R>0$, and $Q_{R,\tht}=B_R(x_0)\times (t_0-\tht,t_0)$, such that $\overline{B_R(x_0)}\subseteq \Om$ and $[t_0-\tht,t_0]\subseteq (0,T)$. Assume that $\delta>0$ is a number such that $\max\{p, 2\}\leq \delta< p\frac{N+2s}{N}$. Let $B_j,\,\tilde{B}_j,\,I_j,\tilde{I_j}$ be as defined in~\cref{seqQ} and~\cref{seqW}. Then we have for all $j\in\NN$
    \begin{align}
        \frac{1}{R^{sp}}\int\limits_{I_{j+1}}\fint\limits_{B_{j+1}} w_{j+1}^\delta(x,t)\,dx\,dt\leq C 2^{bj} \tilde{\mathcal{B}} \left( \frac{\mathcal{A}_k}{R^{sp}} \int\limits_{I_j}\fint\limits_{B_j}w_j^\delta\,dx\,dt \right)^{1+\frac{s\delta}{N\kappa}},
        \label{recest2}
    \end{align} where
    \begin{align}
        \mathcal{A}_k:= & \frac{1}{\tk^{\delta-p}}+\frac{R^{sp-qs'}}{\tk^{\delta-q}}+\frac{R^{sp}}{\tht\tk^{\delta-2}}, b=\left(1+\frac{sp}{N}\right)\left(N+\max\{ps,qs'\}+\delta\right),\mbox{ and,}\\
        \tilde{\mathcal{B}}= & \left[\frac{1}{\sigma^{ps}(1-\sigma)^{N+ps}}+\frac{1}{\sigma^{qs'}(1-\sigma)^{N+qs'}}+\frac{1}{(1-\sigma)^{q}}\right]^{\left(1+\frac{sp}{N}\right)\left(\frac{\delta}{p\kappa}\right)}.
    \end{align}
\end{lemma}

\begin{proof}
    Let $\kappa=1+\dfrac{2s}{N}$. Since $\delta<p\kappa$, we have by H\"older inequality that
    \begin{align}
        \int\limits_{I_{j+1}}\fint\limits_{B_{j+1}} w_{j+1}^\delta(x,t)\,dx\,dt& \leq \int\limits_{I_{j+1}}\fint\limits_{B_{j+1}} \tw_{j}^\delta(x,t)\,dx\,dt\nonumber\\
        &\leq \left(\underbrace{\int\limits_{I_{j+1}}\fint\limits_{B_{j+1}} \tw_{j}^{p\kappa}(x,t)\,dx\,dt}_{(i)}\right)^{\frac{\delta}{p\kappa}}\left(\underbrace{\int\limits_{I_{j+1}}\fint\limits_{B_{j+1}} \chi_{\{u\geq \tk_j\}}(x,t)\,dx\,dt}_{(ii)}\right)^{1-\frac{\delta}{p\kappa}}.\label{est2.5}
    \end{align} 
    Observe that by \cref{improvK}, we have
    \begin{align}
        \label{est3}
        (ii)\leq C \frac{2^{\delta j}}{\tk^\delta}\int\limits_{I_{j}}\fint\limits_{B_{j}} w_{j}^\delta(x,t)\,dx\,dt,\mbox{ and}
    \end{align}
    \begin{align}
        \label{est4}
        \int\limits_{I_{j+1}}\fint\limits_{B_{j+1}} \tw_{j}^p(x,t)\,dx\,dt\leq C \frac{2^{(\delta-p) j}}{\tk^{\delta-p}}\int\limits_{I_{j}}\fint\limits_{B_{j}} w_{j}^\delta(x,t)\,dx\,dt.
    \end{align}
    On the other hand, by Sobolev embedding \cref{sobolev},~\cref{est4}, and \cref{recest1}, we obtain
    \begin{align}
        \label{est5}
        (i)\leq &C R^{sp} \left(\int\limits_{I_{j+1}}\int\limits_{B_{j+1}}\fint\limits_{B_{j+1}}\frac{|\tw_j(x,t)-\tw_j(y,t)|^p}{|x-y|^{N+sp}}\,dx\,dy\,dt+\frac{1}{R^{sp}}\int\limits_{I_{j+1}}\fint\limits_{B_{j+1}}|\tw_j(x,t)|^p\,dx\,dt\right)\nonumber\\
        &\qquad\times\left(\sup_{t\in I_{j+1}}\fint_{B_{j+1}}|\tw_j(x,t)|^2\,dx\right)^{\frac{sp}{N}}\nonumber\\
        &\leq C R^{sp} \left[\frac{2^{(N+sp+\delta-1)j}}{R^{sp}\tk^{\delta-p}}+\frac{2^{(N+qs'+\delta-1)j}}{R^{qs'}\tk^{\delta-q}}+\frac{2^{(\delta-1)j}}{\tht\tk^{\delta-2}}\right]^{1+\frac{sp}{N}}\,\left(\mathcal{B}\int\limits_{I_j}\fint\limits_{B_j}w_j^\delta\,dx\,dt\right)^{1+\frac{sp}{N}}\nonumber\\
        &= C R^{sp} \left[\frac{2^{(N+sp+\delta-1)j}}{\tk^{\delta-p}}+\frac{2^{(N+qs'+\delta-1)j}R^{sp}}{R^{qs'}\tk^{\delta-q}}+\frac{2^{(\delta-1)j}R^{sp}}{\tht\tk^{\delta-2}}\right]^{1+\frac{sp}{N}}\,\left(\frac{\mathcal{B}}{R^{sp}}\int\limits_{I_j}\fint\limits_{B_j}w_j^\delta\,dx\,dt\right)^{1+\frac{sp}{N}}\nonumber\\
        &\leq C 2^{bj} R^{sp} \left( \frac{\mathcal{B}\mathcal{A}_k}{R^{sp}} \int\limits_{I_j}\fint\limits_{B_j}w_j^\delta\,dx\,dt \right)^{1+\frac{sp}{N}},
    \end{align} where
    \begin{align*}
        \mathcal{A}_k&:= \frac{1}{\tk^{\delta-p}}+\frac{R^{sp-qs'}}{\tk^{\delta-q}}+\frac{R^{sp}}{\tht\tk^{\delta-2}},\, b=\left(1+\frac{sp}{N}\right)\left(N+\max\{ps,qs'\}+\delta\right), \mbox{ and }\\
        \mathcal{B}&:= \left[\frac{1}{\sigma^{ps}(1-\sigma)^{N+ps}}+\frac{1}{\sigma^{qs'}(1-\sigma)^{N+qs'}}+\frac{1}{(1-\sigma)^{q}}\right].
    \end{align*}
    Substituting \cref{est3} and \cref{est5} in \cref{est2.5}, we are given
    \begin{align*}
        \frac{1}{R^{sp}}\int\limits_{I_{j+1}}\fint\limits_{B_{j+1}} w_{j+1}^\delta(x,t)\,dx\,dt\leq C 2^{bj}\tilde{\mathcal{B}} \left( \frac{\mathcal{A}_k}{R^{sp}} \int\limits_{I_j}\fint\limits_{B_j}w_j^\delta\,dx\,dt \right)^{1+\frac{s\delta}{N\kappa}},
    \end{align*} which is \cref{recest2} for $\tilde{\mathcal{B}}=\left[\frac{1}{\sigma^{ps}(1-\sigma)^{N+ps}}+\frac{1}{\sigma^{qs'}(1-\sigma)^{N+qs'}}+\frac{1}{(1-\sigma)^{q}}\right]^{\left(1+\frac{sp}{N}\right)\left(\frac{\delta}{p\kappa}\right)}$.
\end{proof}

\subsection{Boundedness estimate}

\begin{proof}(Proof of \cref{mainthm1})
 Let us set 
 \begin{align}
     Y_j:= \frac{1}{R^{sp}}\int\limits_{I_j}\fint\limits_{B_j}w_j^\delta\,dx\,dt.
 \end{align}
 Then \cref{recest2} becomes
 \begin{align}\label{recest3}
     Y_{j+1}\leq C 2^{bj}\tilde{\mathcal{B}} \left( \mathcal{A}_k Y_j \right)^{1+\frac{s\delta}{N\kappa}}.
 \end{align}
 Further taking $\tk >1$ and $\tau:=\min\{\delta-q,\delta-p,\delta-2\}$, we get that 
 \begin{align*}
     \mathcal{A}_k\leq \frac{\mathcal{A}}{\tk^\tau} \mbox{ where } 
     \mathcal{A}:= 1+R^{sp-qs'}+\frac{R^{sp}}{\tht}.
 \end{align*}

 Hence \cref{recest3} becomes
 \begin{align}\label{recest4}
     Y_{j+1}\leq C 2^{bj} \tilde{\mathcal{B}}\left(\frac{\mathcal{A}}{\tk^\tau} Y_j\right)^{1+\frac{s\delta}{N\kappa}}.
 \end{align}

Choose $\tht=R^{sp}$. Let $\tk$ be chosen so that

 \begin{align*}
        \tk \geq \max\Biggl\{ \text{Tail}_{p,s,\infty}(u_+;x_0,R/2,I_0), \text{Tail}_{q,s',\infty}(u_+;x_0,R/2,I_0),\nonumber\\
         C 2^{\frac{bN^2\kappa^2}{s\delta\tau(N\kappa+s\delta)}} \tilde{\mathcal{B}}^{\frac{N\kappa}{\tau(N\kappa+s\delta)}}\mathcal{A}^{\frac{1}{\tau}} \left(\fint\limits_{t_0-R^{sp}}^{t_0}\fint\limits_{B_R(x_0)} u_+^\delta\,dx\,dt \right)^{\frac{s\delta}{\tau(N\kappa+s\delta)}}\wedge 1 
        \Biggr\},
 \end{align*} where the constant $C$ depends on $N,s,p,\delta,s',q$. This choice guarantees by \cref{geo_con} that
 \begin{align*}
     \underset{Q_{\frac{R}{2},\frac{R^{sp}}{2}}}{\text{ess sup}} \, u  \,\leq\, &\text{Tail}_{p,s,\infty}(u_+;x_0,R/2,t_0-R^{sp},t_0) + \text{Tail}_{q,s',\infty}(u_+;x_0,R/2,t_0-R^{sp},t_0)\\
     & + C 2^{\frac{bN^2\kappa^2}{s\delta\tau(N\kappa+s\delta)}} \tilde{\mathcal{B}}^{\frac{N\kappa}{\tau(N\kappa+s\delta)}}\mathcal{A}^{\frac{1}{\tau}} \left(\fint\limits_{t_0-R^{sp}}^{t_0}\fint\limits_{B_R(x_0)} u_+^\delta \,dx\,dt \right)^{\frac{s\delta}{\tau(N\kappa+s\delta)}}\wedge 1,
 \end{align*} which is~\cref{boundest}.
\end{proof}

\section{Limit case}\label{sec5}

In this section, we consider the limit case $q=p_*$ in order to complete the proof of \cref{mainthm1}. We only indicate the modifications required. In this case, we choose the sequence $k_j$ as
\begin{align}
    k_j:=\tk\left(1-\frac{1}{2^{j+1}}\right),\mbox{ and } Y_j:= \frac{1}{R^{sp}}\int\limits_{I_j}\fint\limits_{B_j}w_j^q\,dx\,dt.
\end{align} while keeping the rest of the sequences $R_j, \tht_j$ the same as before. 

By the same calculations as before, we will find that the iterative inequality \cref{recest4} becomes

\begin{align}
    Y_{j+1}\leq C 2^{bj}\tilde{\mathcal{B}} \left(\mathcal{A} Y_j\right)^{1+\frac{s\delta}{N\kappa}},
\end{align} and hence the dependence on $\tk$ is only in $Y_j$. 

In order to apply \cref{geo_con}, we need $Y_0$ to be sufficiently small, however since we have redefined the sequence $k_j$, we have
\begin{align}
    Y_0= \fint_{t_0-R^{sp}}^{t_0}\fint_{B_R(x_0)} \left(u-\frac{\tk}{2}\right)_+^q\,dx\,dt
\end{align}
which can be made as small as we like provided $\tk$ is big enough. Hence, we get boundedness of $u$ in this case, but without an explicit bound. This completes the proof of \cref{mainthm1}.

\section{Subcritical Case}\label{sec6}

The proof of \cref{mainthm2} will be completed in four steps. If \descref{H1}{H1} holds, then we begin by proving qualitative boundedness using the higher integrability of $u$. Then, we prove an improved estimate by exploiting the local boundedness of $u$, subsequently, we perform a third iteration to get an estimate only in terms of the $L^r$ norm of $u$. In the fourth step, we deal with the limiting case $q=p_*$. If \descref{H2}{H2} holds, then we do not need to perform Step 1, we can directly go to steps $2$, $3$, and $4$.

\begin{proof}[Proof of \cref{mainthm2}]

    \underline{Step 1.} Assume that \descref{H1}{H1} holds. In this step, we will prove the qualitative fact that $u$ is locally bounded. We will assume that $q<p_*$. The limiting case is dealt with in Step 4. Notice that $p_*=p\frac{2s+N}{N}\leq 2$. As before, let $(x_0,t_0)\in \Om_T$, $R>0$, and $Q_{R,\tht}=B_R(x_0)\times (t_0-\tht,t_0)$, such that $\overline{B_R(x_0)}\subseteq \Om$ and $[t_0-\tht,t_0]\subseteq (0,T)$. Define the quantities $R_j, \tht_j, \tilde{R}_J,\tilde{\tht}_j, \tilde{k},k_j,w_j$ and the cylinders $Q_j, \tilde{Q}_j$ as before, as in \cref{seqR}, \cref{seqQ}, \cref{choiceofk}, \cref{seqW}. In this step, we will choose $\tht=R^{sp}$ and $\sigma=1/2$. Let $\alpha\in (0,1)$ be such that $1=\frac{2\alpha}{q}+\frac{2(1-\alpha)}{r}$.

    Define \begin{align*}
        Y_j:=\fint\limits_{I_j}\fint\limits_{B_j} w_j^2\,dx\,dt\mbox{ and }M:=\left(\fint\limits_{I_j}\fint\limits_{B_j}|u|^r\,dx\,dt\right)^{2\frac{1-\alpha}{r}}.
    \end{align*}

    Then, 
    \begin{align}
    Y_{j+1}&\leq \fint\limits_{I_{j+1}}\fint\limits_{B_{j+1}} \tilde{w}_j^2\,dx\,dt\nonumber\\
    &\leq \left(\fint\limits_{I_{j+1}}\fint\limits_{B_{j+1}} \tilde{w}_j^q\,dx\,dt\right)^{\frac{2\alpha}{q}}\left(\fint\limits_{I_{j+1}}\fint\limits_{B_{j+1}} \tilde{w}_j^r\,dx\,dt\right)^{\frac{2(1-\alpha)}{r}}\nonumber\\
    &= M \left(\fint\limits_{I_{j+1}}\fint\limits_{B_{j+1}} \tilde{w}_j^q\,dx\,dt\right)^{\frac{2\alpha}{q}}\nonumber\\
    &\leq M \left\{\left(\underbrace{\fint\limits_{I_{j+1}}\fint\limits_{B_{j+1}} \tw_{j}^{p\kappa}(x,t)\,dx\,dt}_{(i)}\right)^{\frac{q}{p\kappa}}\left(\underbrace{\fint\limits_{I_{j+1}}\fint\limits_{B_{j+1}} \chi_{\{u\geq \tk_j\}}(x,t)\,dx\,dt}_{(ii)}\right)^{1-\frac{q}{p\kappa}}\right\}^{\frac{2\alpha}{q}}.\label{estsub1}
    \end{align}
    
    By a calculation exactly similar to the one performed to obtain \cref{est5}, we obtain 
    \begin{align}\label{estsub2}
        (i)\leq C 2^{\tilde{b} j}\left(\mathcal{B}\fint\limits_{I_j}\fint\limits_{B_j} w_j^2\,dx\,dt\right)^{1+\frac{sp}{N}},
    \end{align} where $\mathcal{B}:=2+R^{sp-qs'}$ and $\tilde{b}:=\kappa(N+\max\{sp,qs'\}+2)$.
    On the other hand, by \cref{improvK}, we have
    \begin{align}
        \label{estsub3}
        (ii)\leq C \frac{2^{2 j}}{\tk^2}\fint\limits_{I_{j}}\fint\limits_{B_{j}} w_{j}^2(x,t)\,dx\,dt.
    \end{align}
    Substituting \cref{estsub2} and \cref{estsub3} in \cref{estsub1}, we obtain
    \begin{align*}
        Y_{j+1}\leq M\left\{ \left( C 2^{\tilde{b} j}\mathcal{B} Y_j \right)^{\left(1+\frac{sp}{N}\right)\frac{q}{p_*}} \left( \frac{C 2^{2j}}{\tk^2} Y_j\right)^{1-\frac{q}{p_*}} \right\}^{\frac{2\alpha}{q}}\leq \frac{CM 2^{\tilde{c}j}}{\tk^{2\left(\frac{p_*-q}{p_*}\right)\left(\frac{2\alpha}{q}\right)}} Y_j^{\left(1+\frac{sq}{\kappa N}\right)\left(\frac{2\alpha}{q}\right)}.
    \end{align*}
    In order to apply \cref{iterlemma}, we need $\gamma:=\left(1+\frac{sq}{\kappa N}\right)\left(\frac{2\alpha}{q}\right)-1>0$ which is justified when \cref{restiction3} of \descref{H1}{H1} holds. Therefore, taking $\tk$ large enough, we obtain boundedness of $u$. However, the quantitative bound obtained in this manner is not optimal in the exponents. Therefore, in subsequent steps, we will do further iterations. When the hypothesis \descref{H2}{H2} holds, we do not require this step.

    \underline{Step 2.} This step is common to both \descref{H1}{H1} and \descref{H2}{H2}. However, the value of the exponent $r$ is to be understood contextually. Moreover, for \descref{H2}{H2}, we simply write $u$ instead of $u_k$.

    Now, for any $\sigma\in [1/2,1)$ and for any number $\rho$ satisfying $\sigma R\leq\sigma \rho<\rho\leq R$, we define the following sequences
\begin{align}\label{seqRsub}
    \rho_j=\sigma\rho+(1-\sigma)\rho 2^{-j},\, j=0,1,\ldots
\end{align} so that $\rho_0=\rho$, $\rho_\infty=\sigma \rho$. Corresponding to these sequences, we define a sequence of nested cylinders
\begin{align}\label{seqQsub}
    &Q_j(z_0)=B_j\times I_j = B_{\rho_j}(x_0)\times [t_0-\rho_j^{sp},t_0]
\end{align}
We also denote $\tilde{\rho_j}:=\frac{\rho_j+\rho_{j+1}}{2}$ so that we have the inclusions \[Q_{j+1}\subset \tilde{Q_j}:=B_{\tilde\rho_j}(x_0)\times [t_0-\tilde\rho_j^{sp},t_0]\subset Q_j.\]

For 
\begin{align}\label{choiceofksub}
    \tk\geq \max\left\{\frac{\text{Tail}_{p,s,\infty}(u_+;x_0,\sigma\rho,t_0-\rho^{sp},t_0)}{2},\frac{\text{Tail}_{q,s',\infty}(u_+;x_0,\sigma\rho,t_0-\rho^{sp},t_0)}{2}\right\}
\end{align}
to be chosen later, we define a sequence of increasing levels as
\begin{align}\label{seqWsub}
    k_j=(1-2^{-j})\tk,\,\tk_j=\frac{k_j+k_{j+1}}{2},\,j=0,1,2,\ldots\nonumber\\
    w_j=(u-k_j)_+,\,\tilde{w}_j=(u-\tk_j)_+,\,j=0,1,2,\ldots
\end{align}

    Define $Z_j:=\fint_{I_j}\fint_{B_j}w_j^r\,dx\,dt$. Then, by a calculation similar to the one performed to obtain \cref{est5}, we get 
    \begin{align*}
        Z_{j+1}&\leq \fint_{I_{j+1}}\fint_{B_{j+1}} \tw_j^r\,dx\,dt\\
        &\leq ||u||_{L^\infty(Q_{\rho,\rho^{sp}}(z_0))}^{r-p_*}\fint_{I_{j+1}}\fint_{B_{j+1}}\tw^{p_*}\,dx\,dt\\
        &\leq C 2^{b'j}||u||_{L^\infty(Q_{\rho,\rho^{sp}}(z_0))}^{r-p_*}\left(\frac{\mathcal{A}}{\tk^\tau} Z_j\right)^{1+\frac{sp}{N}},
    \end{align*} where
    \begin{align*} 
        \mathcal{A}:=&(2+R^{sp-qs'})\left[\frac{1}{\sigma^{ps}(1-\sigma)^{N+ps}}+\frac{1}{\sigma^{qs'}(1-\sigma)^{N+qs'}}+\frac{1}{(1-\sigma)^{q}}\right],\\
        b'=&\left(1+\frac{sp}{N}\right)\left(N+\max\{ps,qs'\}+r\right),\mbox{ and,}\\
        \tau=&\min\{r-p,r-q,r-2\}.
    \end{align*}

By an application of \cref{iterlemma}, we get 
\begin{align}\label{subest7}
    \underset{Q_{\sigma\rho,(\sigma\rho)^{sp}}}{\text{ess sup}} \, u  \,\leq\, &\frac{\text{Tail}_{p,s,\infty}(u_+;x_0,\sigma\rho,t_0-\rho^{sp},t_0)}{2}+ \frac{\text{Tail}_{q,s',\infty}(u_+;x_0,\sigma\rho,t_0-\rho^{sp},t_0)}{2}\nonumber\\
    &\qquad + C' 2^{\frac{b'N}{(r-2)(N+sp)}}\mathcal{A}^{\frac{sp}{(r-2)N}}\left(\fint\fint\limits_{Q_{\rho,\rho^{sp}}}u^r\,dx\,dt\right)^{\frac{sp}{(r-2)(N+sp)}} \norm{u}_{L^\infty(Q_{\rho,\rho^{sp}})}^{\frac{N(r-p_*)}{(r-2)(N+sp)}}\nonumber\\
    \,\leq\, &\frac{\text{Tail}_{p,s,\infty}(u_+;x_0,R/2,t_0-R^{sp},t_0)}{2}+ \frac{\text{Tail}_{q,s',\infty}(u_+;x_0,R/2,t_0-R^{sp},t_0)}{2}\nonumber\\
    &\qquad + C' 2^{\frac{b'N}{(r-2)(N+sp)}}\mathcal{A}^{\frac{sp}{(r-2)N}}\left(\fint\fint\limits_{Q_{R,R^{sp}}}u^r\,dx\,dt\right)^{\frac{sp}{(r-2)(N+sp)}} \norm{u}_{L^\infty(Q_{\rho,\rho^{sp}})}^{\frac{N(r-p_*)}{(r-2)(N+sp)}}.
\end{align}

\underline{Step 3.} 

It remains to push the $L^\infty$ norm to the left in the estimate above. Once more, the numbers $\rho$ and $R$ is arbitrary such that $Q_{\rho,\rho^{sp}}(z_0)\subset Q_{R,R^{sp}}(z_0)\Subset \Om_T$.

Therefore, we can write \cref{subest7} for a sequence $\rho_n=R/2+\sum_{i=0}^n 2^{-i-2} R$ with $M_n:=\norm{u}_{L^\infty(Q_{\rho_n,\rho_n^{sp}})}$.  

To conclude the result, we shall apply \cref{interlope_lemma} with 
\begin{align*}
    \mathcal{K}=& \frac{\text{Tail}_{p,s,\infty}(u_+;x_0,R/2,t_0-R^{sp},t_0)}{2}+ \frac{\text{Tail}_{q,s',\infty}(u_+;x_0,R/2,t_0-R^{sp},t_0)}{2},\\
    C=& C'2^{\frac{b'}{(r-2)(N+sp)}}\left(\fint\fint\limits_{Q_{R,R^{sp}}}u^r\,dx\,dt\right)^{\frac{sp}{(r-2)(N+sp)}},\\
    1-\alpha=&\frac{N(r-p_*)}{(r-2)(N+ps)},\\
    b=&2^{\frac{sp(N+ps+qs'+q)}{(r-2)N}},
\end{align*} and use the fact that \[r>\frac{N(2-p)}{p}\] is equivalent to \[\frac{N(r-p_*)}{(r-2)(N+ps)}<1.\] 

We have also used the fact that \[\sigma=\frac{\rho_n}{\rho_{n+1}}=1-\frac{1}{2^{n+2}+2^{n+1}-1}.\] so that
\[\frac{1}{1-\sigma}=2^{n+1}+2^{n+2}-1<2^{n+3}.\]
As a result, we obtain the inequality
\begin{align*}
    \underset{Q_{R/2,(R/2)^{sp}}}{\text{ess sup}} \, u  \,\leq\, &\text{Tail}_{p,s,\infty}(u_+;x_0,R/2,t_0-R^{sp},t_0) + \text{Tail}_{q,s',\infty}(u_+;x_0,R/2,t_0-R^{sp},t_0)\nonumber\\
    &\qquad\qquad+C\left(\fint\fint\limits_{Q_{R,R^{sp}}}u^r\,dx\,dt\right)^{\frac{sp}{(r-2)(N+sp)-N(r-p_*)}},
\end{align*} where the constant $C$ depends on $N, s, s', p, q, r, M$. This completes the proof of the boundedness estimate when $q<p_*$.

\underline{Step 4.} We omit the proof in the limit case of $q=p_*$ since it is the same as in \cref{sec5}.
\end{proof}

\section*{Acknowledgments} 
The authors would like to thank Karthik Adimurthi for introducing us to this subject and for illuminating discussions. The authors were supported by the Department of Atomic Energy,  Government of India, under	project no.  12-R\&D-TFR-5.01-0520. 

\bibliography{MyLibrary}

\end{document}